\documentclass[11pt, reqno]{amsart}
\usepackage{amssymb}
\usepackage{verbatim}
\usepackage[vcentermath]{youngtab}
\usepackage{float}
\usepackage{times}
\usepackage[T1]{fontenc}
\usepackage{mathrsfs}
\usepackage{epsfig}
\usepackage{color}
\usepackage{enumitem}

\usepackage{times}
\usepackage[T1]{fontenc}
\usepackage{mathrsfs}
\usepackage{epsfig}
\usepackage{color}

\newtheorem{thm}{Theorem}[section]

\newtheorem{prop}[thm]{Proposition}

\newtheorem{defi}[thm]{Definition}

\theoremstyle{remark}

\newtheorem*{remark}{Remark}

\theoremstyle{definition}

\newtheorem{example}[thm]{Example}

\numberwithin{equation}{section}

\def\ZZ{\mathbb Z}
\def\FF {\mathbb F}
\def\CC{\mathbb C}
\def\QQ{\mathbb Q}
\def\RR{\mathbb R}
\def\cB{\mathcal B}

\def\sign{\mathrm{sgn}}
\def\Sign{\mathrm{\mathbf{Sgn}}}
\def\sS{\mathcal S}
\def\triv{\mathbf{1}}
\def\conf{\mathrm{Conf}}
\def\pconf{\mathrm{PConf}}

\def\PP{\mathbb P}
\def\cC{\mathcal C}
\def\im{\mathrm{Im}}

\def\part{\mathrm{Par}}
\def\sA{\mathcal{A}}

\def\sM{\mathcal{M}}

\def\Ind{\mathrm{Ind}}

\def\w{t}
\def\tr{\mathrm{Trace}}

\begin{document}



\title{Polynomial Splitting Measures and Cohomology of the Pure Braid Group}

\author{Trevor Hyde}
\address{Dept. of Mathematics\\
University of Michigan \\
Ann Arbor, MI 48109-1043\\
}
\email{tghyde@umich.edu}

\author{Jeffrey C.  Lagarias}
\address{Dept. of Mathematics\\
University of Michigan \\
Ann Arbor, MI 48109-1043\\
}
\email{lagarias@umich.edu}


\thanks{Work of the  second author was partially supported by NSF grant DMS-1401224.}

\date{February 24, 2017}
 
\begin{abstract}
We study for each $n$  a one-parameter family of complex-valued measures
on the symmetric group $S_n$, which interpolate the probability of a monic, degree $n$, square-free polynomial in $\FF_q[x]$ having a given factorization type.
For a fixed factorization type, indexed by a partition $\lambda$ of $n$, the measure is
known to be  a Laurent polynomial.
We express the coefficients of this polynomial in terms of characters associated to 
$S_n$-subrepresentations of the cohomology of the pure braid group $H^{\bullet}(P_n, \QQ)$.
We deduce that the splitting measures for all parameter values $z= -\frac{1}{m}$ \big(resp. $z= \frac{1}{m}$\big), after rescaling, are characters
of $S_n$-representations (resp. virtual $S_n$-representations.)
\end{abstract}

\maketitle


\section{Introduction}\label{sec1}

The purpose of this paper is to study for each $n \ge 1$ a one-parameter family of complex-valued measures on
the symmetric group $S_n$ arising from a problem in number theory, and to exhibit  
an explicit representation-theoretic connection between these  measures and  the characters of the natural $S_n$-action
on the rational cohomology of the pure braid group $P_n$.

This family of  measures, denoted $\nu_{n, z}^{\ast}$,
was introduced by the second author and B. Weiss in \cite{Lagarias-W:2015},
where they were called {\em $z$-splitting measures}, with parameter $z$.
The  measures interpolate from  prime power values $z=q$ the probability of a monic, degree $n$, square-free polynomial 
in $\FF_q[x]$ having a given factorization type.
Square-free factorization types are indexed by partitions $\lambda$ of  $n$ specifying the degrees of the irreducible factors.
Each partition $\lambda$ of $n$ corresponds to a conjugacy class $C_{\lambda}$ of the symmetric group $S_n$;
distributing the probability of a factorization of type $\lambda$ equally across the elements of $C_{\lambda}$ defines a
probability measure  on $S_n$.
A key property of the resulting  probabilities   is that for a fixed partition $\lambda$, their values are described by a rational function 
 in the size of the field $\FF_q$ as $q$ varies.  This property permits  interpolation  from $q$ to a  parameter $z \in \PP^{1}(\CC)$
 on the Riemann sphere, to obtain a family of complex-valued measures $\nu_{n, z}^{\ast}$ on $S_n$ given in  Definition \ref{de23} below.

On the number theory side, these  measures connect with problems on the splitting of ideals in $S_n$-number fields, which are degree $n$ number fields
formed by adjoining a root of a degree $n$ polynomial over $\ZZ[x]$ whose splitting field has Galois group $S_n$.
The paper  \cite[Theorem 2.6] {Lagarias-W:2015} observed that for  primes $p < n$ these measures vanish on certain conjugacy classes,
corresponding to the phenomenon of essential discriminant divisors of polynomials having Galois group $S_n$, first noted by Dedekind \cite{Dedekind:1878}
in 1878.  These measures converge to the uniform measure on the symmetric group as $z= p \to \infty$, and in this limit agree with
a conjecture of Bhargava \cite[Conjecture 1.3]{Bhargava:2007} on the distribution of splitting types  of the prime $p$ in $S_n$-extensions of discriminant $|D| \le B$
as the bound $B \to \infty$, conditioned on $(D, p)=1$.

The second author subsequently studied these measures interpolated
at the special value $z = 1$, viewed as representing splitting probabilities for polynomials 
over the (hypothetical)  ``field with one element $\FF_1$'' \cite{Lagarias:2016}.
These measures, called {\em $1$-splitting measures}, turn out to be
  signed measures for all $n \ge 3$.
They are 
supported on a small set of conjugacy classes, the  Springer regular elements of $S_n$ which are those conjugacy
classes $C_{\lambda}$ for which $\lambda$ has a rectangular Young diagram or a rectangle plus a single box.
Treated as class functions on $S_n$, rather than as measures, they 
were found  to have a representation-theoretic interpretation: 
after rescaling by $n!$, the 1-splitting measures are virtual characters of $S_n$ corresponding to explicitly determined representations. 
As $n$ varies, their values on conjugacy classes were observed to have arithmetic properties compatible
with the multiplicative structure of $n$; letting $n = \prod_{p} p ^{e_p}$ be the prime factorization
of $n$,   the value of the  measure on each conjugacy class factors 
as a product of  
values on  classes
of  smaller symmetric groups $S_{p^{e_p}}$.
That paper also showed the rescaled $z$-splitting  measures at $z=-1$
 have a related representation-theoretic interpretation.

In this paper we extend the representation-theoretic interpretation to the entire family of $z$-splitting measures and
 relate it to the cohomology of the pure braid group.
Our starting point is the observation made in \cite[Lemma 2.5]{Lagarias:2016}
 that for a fixed conjugacy class 
the $z$-splitting measures
are Laurent polynomials in $z$.
They have  
degree at most $n-1$, so may  be written
\[
	\nu_{n, z}^{\ast} (C_{\lambda}) = \sum_{k=0}^{n-1}  \alpha_n^k(C_{\lambda})\big(\tfrac{1}{z}\big)^k,
\]
with rational coefficients $\alpha_n^k(C_{\lambda})$, where $\lambda$ is a partition of $n$.
We call the $\alpha_n^k(C_{\lambda})$ {\em splitting measure coefficients.}
A main observation  of this paper is that  
 each splitting measure coefficient $\alpha_n^k(C_{\lambda})$, viewed as a function of $\lambda$,
 is a rescaled character $\chi_n^k$ of a certain $S_n$-subrepresentation $A_n^k$ of the cohomology of the pure braid group $H^k(P_n, \QQ)$.
{ The pure braid groups $P_n$ and their cohomology, along with  the subrepresentations $A_n^k$,
are  defined
and discussed in Section \ref{sec:braid-group}.}
{In Section \ref{sec:43a} we  identify the $S_n$-representation  $A_n^k$ with the cohomology of a complex manifold $Y_n$ carrying an $S_n$-action. 
We deduce as a consequence  a topological interpretation of the $1$-splitting measure  as a rescaled version of the $S_n$-equivariant Euler characteristic of $Y_n$.}
We also deduce that the rescaled $z$-splitting measure is a character of $S_n$ 
at $z = -\frac{1}{m}$ and  is a virtual character of $S_n$ at $z = \frac{1}{m}$,  for all integers $m\geq 1$.

The last result extends the representation-theoretic connection of \cite{Lagarias:2016} for $z= \pm 1$ to parameters $z= \pm \frac{1}{m}$
for all $m \ge 1$.


\subsection{Results}\label{sec11}

The  {\em $z$-splitting measure} on a conjugacy class $C_{\lambda}$ of $S_n$ is the rational function of $z$
\[
	\nu_{n, z}^{\ast} (C_{\lambda}) := \frac{N_{\lambda}(z)}{ z^n - z^{n-1}},
\]
where  $N_{\lambda}(z) \in \QQ[z]$ denotes
the {\em cycle polynomial} associated to a partition
$\lambda$ describing the cycle lengths of $C_\lambda$.
Given  $\lambda= \big(1^{m_1(\lambda)} 2^{m_2(\lambda)} \cdots n^{m_n(\lambda)} \big)$,
the  associated cycle polynomial is 
\begin{equation}\label{eq:11}
	N_{\lambda}(z) :=  \prod_{j\geq 1} {{M_j(z)}\choose{m_j(\lambda)}},
\end{equation}
where $M_j(z)$ denotes the $j$th necklace polynomial.   
The  \emph{necklace polynomial} $M_j(z)$ of order $j$ is given by 
\[
	M_j(z) := \frac{1}{j}  \sum_{d \mid j} \mu(d) z^{j/d}.
\]
where  $\mu(d)$ is the M\"{o}bius function. 

To avoid confusion we make a  remark on values of measures.
Given a class function $f$ on $S_n$ we write $f(C_{\lambda})$ to mean the sum of the values of $f$ on $C_\lambda$,
and write $f(\lambda)$ to mean  the value $f(g)$
taken at one element $g \in C_{\lambda}$; the latter notation is standard for characters.
Thus $\nu_{n,z}^{\ast}(C_{\lambda}) = |C_\lambda| \nu_{n, z}^{\ast}(\lambda)$.

In Section \ref{sec:rep-interpretation} we express the coefficients of 
the family of cycle polynomials $N_\lambda(z)$  in terms of characters of the cohomology of the pure braid group 
$P_n$ viewed as an $S_n$-representation. 
%
%
\begin{thm}[Character interpretation of cycle polynomial coefficients]\label{thm:main-0}
Let $\lambda$ be a partition of $n$
and $N_\lambda(z)$ be a cycle polynomial. 
Then
\[
    N_\lambda(z) = \frac{|C_{\lambda}|}{n!}\sum_{k=0}^n{(-1)^k h_n^k(\lambda) z^{n-k}}.
\]
where $h_n^k$ is the character of the $k$th cohomology of the pure braid group $H^k(P_n,\QQ)$, viewed as an $S_n$-representation.
\end{thm}

Theorem \ref{thm:main-0} is  a rescaled version of  a result of Lehrer \cite[Theorem 5.5]{Lehrer:1987}.
Lehrer arrived at it from his study of the Poincar\'{e} polynomials associated to the elements of a Coxeter group
acting on the complements of certain complex hyperplane arrangements. 
We arrived at it through a direct study of the cycle polynomial $N_{\lambda}(z)$
appearing in the definition of the $z$-splitting measure, relating it to representation stability
using the twisted Grothendieck-Lefschetz formula
of Church, Ellenberg, and Farb \cite[Prop. 4.1]{CEF:2014}.
We include a proof of Theorem \ref{thm:main-0} (as Theorem \ref{cycle-coeffs});
the method behind this  proof also traces back to  work of Lehrer
 \cite{Lehrer:1992}.

 At the end of Section \ref{sec:rep-interpretation} we  apply Theorem \ref{thm:main-0}
together with the formula  \eqref{eq:11} for $N_{\lambda}(z)$
to   obtain explicit expressions for various   
 characters $h_n^k$ showing number-theoretic structure, and to determine restrictions on the support
 of various $h_n^k$.

In Section \ref{sec:braid-group} we review Arnol'd's presentation of the cohomology ring
of the pure braid group.
In Section \ref{sec:N41} we use it  derive an exact
sequence determining  certain $S_n$-subrepresentations $A_n^k$ of $H^k(P_n,\QQ)$
which play the main role in our results.  These subrepresentations lead to a direct sum
decomposition $H^k(P_n, \QQ) \simeq A_n^{k-1}\oplus A_n^k$, for each $k \ge 0$.
 In Section \ref{sec:43a} we  interpret the $A_n^k$
 as the cohomology   of an $(n-1)$-dimensional complex manifold  $Y_n$ that carries an $S_n$-action. 
The manifold $Y_n$  is  the quotient of  the pure configuration space  $\pconf_n(\CC)$  of $n$ distinct 
(labeled) points in $\CC$ by a free action of  $\CC^{\times}$.
 
The  main result of this paper, given in Section \ref{sec:splitting-characters}, expresses  the $z$-splitting measures $\nu_{n, z}^{\ast}$
in terms of the characters $\chi_n^k$ of the $S_n$-representations $A_n^k$.

%
%
\begin{thm}[Character interpretation of splitting measure coefficients]\label{thm:main-1}
For each $n \ge 1$ and $0 \le k \le n-1$ there is an $S_n$-subrepresentation $A_n^k$ of $H^k(P_n,\QQ)$ (constructed explicitly in Proposition \ref{prop:exact-lemma}) 
with  character $\chi_n^k$ 
such that for each partition $\lambda$ of $n$,
\[
	\nu_{n,z}^\ast(C_\lambda) = \frac{|C_{\lambda}|}{n!} \sum_{k=0}^{n-1}\chi_n^k(\lambda)\big(-\tfrac{1}{z}\big)^k.
\]
Thus the splitting measure coefficient $\alpha_n^k(C_{\lambda})=|C_{\lambda}|  \,\alpha_n^k(\lambda)$ is given by
\[
	 \alpha_n^k(C_{\lambda})  =(-1)^k \frac{|C_{\lambda}|}{n!}  \chi_n^{k}(\lambda).
\]
\end{thm}

In Section \ref{sec:hyperplane2} we interpret this result in terms of 
cohomology of the manifold $Y_n$.
On setting $t = -\frac{1}{z}$, we have that  for each $g \in S_n$,
\[
	\nu_{n, z}^{\ast}(g) = \frac{1}{n!} \sum_{k=0}^{n-1} \tr(g, H^k(Y_n, \QQ)) t^k,
\]
which is a value of the  equivariant Poincar\'{e} polynomial for $Y_n$ with respect to the $S_n$-action (Theorem \ref{thm:splitting-coeffs2}).
{In particular we obtain the following  topological interpretation
of the  $1$-splitting measure, as  the   special case $t=-1$.}

%
%
\begin{thm}[Topological interpretation of $1$-splitting measure]\label{thm:cor-main-1a} 
Let $Y_n$ denote the open complex manifold  $\pconf_n(\CC)/\CC^{\times}$,
which carries an $S_n$-action under permutation of the $n$ points. 
Then the rescaled $1$-splitting measure $\nu^*_{n,1}(\cdot)$ evaluated
at   elements $g \in S_n$ 
is the equivariant Euler characteristic of  $Y_n$,
\[
\nu^*_{n,1}(g) =  \frac{1}{n!}\sum_{k=0}^{n-1}{ (-1)^k \tr(g, H^k( Y_n, \QQ))},
\]
with respect to its $S_n$-action.
\end{thm}

In  Section \ref{sec:52} we obtain another corollary of Theorem \ref{thm:main-1}.
For  $z = -\frac{1}{m}$ with $m \ge 1$, the rescaled splitting measure
$\frac{n!}{|C_\lambda|}\nu_{n,z}^{k} (C_\lambda)$ is the character of an $S_n$-representation, and when $z= \frac{1}{m}$
it  is the character of a virtual $S_n$-representation (Theorem \ref{thm:main-2a}).

In Section \ref{sec:53} we deduce  an interesting consequence  concerning the $S_n$-action on  
the full cohomology ring $H^{\bullet}(P_n, \QQ)$. 
 {The structure of the cohomology ring of the pure braid group $H^{\bullet}(P_n, \QQ)$ 
as an $S_n$-module has an extensive literature.
Orlik and Solomon \cite{OrlikS:1980} noted that  $H^{\bullet}(P_n, \QQ) \simeq H^{\bullet} (M(\sA_n), \QQ)$ as $S_n$-modules, where 
$$M(\sA_n) = \CC^n \smallsetminus \cup_{H \in \sA_n} H$$
is the complement of the (complexified) braid arrangement $\sA_n$, i.e. the arrangement of $n(n-1)/2$ hyperplanes $z_i = z_j$ 
in $\CC^n$ where $1\leq i < j \leq n$ are the coordinate functionals of $\CC^n$. }
The structure of the cohomology groups $H^k(M(\sA_n), \CC) = H^k(M(\sA_n),\QQ)\otimes \CC$ as  $S_n$-representations was determined 
 in 1986 by Lehrer and Solomon \cite[Theorem 4.5]{LS:1986}  in terms of induced representations
$\Ind_{Z(C_\lambda)}^{S_n}(\xi_{\lambda}) $ for  specific linear representations $\xi_{\lambda}$ on the centralizers $Z(C_{\lambda})$ of 
conjugacy classes $C_{\lambda}$ having $n-k$ cycles. 
In 1987 Lehrer \cite[p.  276]{Lehrer:1987} noted that  his  results on Poincar\'{e} polynomials
  implied the ``curious consequence'' that the 
action of $S_n$ on $\bigoplus_k H^{k}(M(\sA_n, \CC))$ is ``almost'' the regular representation in the sense that the dimension is $n!$ and the character $\theta(g)$
of this representation is $0$ unless $g$ is the identity element or a transposition, see 
also \cite[Corollary (5.5)', Prop. (5.6)]{Lehrer:1987}. 
where $r$ is a reflection and $1$ is the trivial representation.
In  Section \ref{sec:53} we apply
Theorem \ref{thm:main-1} together with
 values of the $(-1)$-splitting measure computed in \cite{Lagarias:2016} to 
 make  a precise connection  between the 
$S_n$-representation structure  on pure braid group cohomology and the regular representation $\QQ[S_n]$.

%
%
\begin{thm}\label{thm:main-2}
Let $\triv_n$,  $\Sign_n$, and $\QQ[S_n]$ be the trivial, sign, and regular representations of $S_n$ respectively. Then there is an isomorphism of $S_n$-representations,
\begin{equation*}
\label{reg isom}
	\bigoplus_{k=0}^n{H^k(P_n,\QQ)\otimes \Sign_n^{\otimes k}} \cong \QQ[S_n].
\end{equation*}
Here $\Sign_n^{\otimes k}\cong \triv_n$ or $\Sign_n$ according to whether $k$ is even or odd.
\end{thm}

When combined   with Lehrer's \cite[Prop. 5.6 (i)]{Lehrer:1987} determination 

{of the character $\theta$
as $2 \,\Ind_{\langle \tau \rangle}^{S_n}(1)$, where $\tau$ is a transposition,}
this result implies  that  each of the characters of the $S_n$-representations acting on the even-dimensional cohomology, resp. odd-dimensional cohomology are  supported on the identity element plus transpositions. We comment on other related work in Section \ref{sec:12}.

{ In Section \ref{sec:stability} we describe further interpretations of the representations $A_n^k$
in terms of other combinatorial homology theories. 
 For fixed $k$ and varying $n$, the sequence of $S_n$-representations $H^{k} (P_n, \QQ)$ was one of
 the basic examples exhibiting 
\emph{representation stability} in the sense of Church and Farb \cite{CF:2013}, see \cite{CEF:2014}, \cite{CEF:2015}).}
We show  in Proposition \ref{prop:connections} that the  representations $A_n^k$ are isomorphic to others appearing in the literature
known to exhibit representation stability.  Hersh and Reiner \cite[Corollary 5.4]{HerRein:2015} { determine the precise rate of stabilization of these representations,}
yielding the following result.

%
%
\begin{thm}[Representation stability for $A_n^k$]\label{thm:main-3a} 
For each fixed $k \ge 1$,
the sequence of $S_n$-representations  $A_n^k$ 
with characters $\chi_n^k$ are representation stable, and stabilize sharply at $n=3k+1$. 
\end{thm}

To summarize these results:
\begin{itemize}
\item[(i)]
We start from {a construction in} number theory:  a set of probability measures on $S_n$ that
         {describe the distribution of degree $n$ squarefree monic polynomial factorizations $\pmod{p}$} defined for a parameter  $z$  being a prime $p$.  These measure values
         interpolate at each fixed  $g \in S_n$ in the $z$-variable as polynomials in
         $1/z$  to define complex-valued measures  on $S_n$.
\item[(ii)]
We make a connection {of the interpolated measures as functions of $z$}  to topology and representation theory:
For fixed $n$  the $k$th Laurent coefficients of the $z$-parametrization at $g \in S_n$ 
(rescaled by $n!$) coincide with the  character of an
$S_n$-subrepresentation $A_n^k$
  of  the cohomology of the pure braid group $P_n$, which is an $S_n$-representation on the cohomology of
{ the complex manifold} $Y_n= \pconf_n(\CC)  / \CC^{\times}$. As $n$ varies with $k$ fixed these
   coefficients { exhibit}  representation stability as $n \to \infty.$
   \item[(iii)]
We deduce that  (rescaled) measure values  at values  $z= -\frac{1}{m}$ for  $m \ge 1$ coincide with characters of 
          certain $S_n$-representations; those at $z= \frac{1}{m}$ with $m \ge 1$ coincide with certain virtual $S_n$-representations.  
           For each $n$ these  representations  combine
          stable and unstable cohomology of $P_n$. 

    \item[(iv)]As a by-product  we find a  precise connection between the (total) cohomology of the pure braid group
as an $S_n$-representation and the regular representation of $S_n$.

\end{itemize}
{The main observation of this paper is the  relation of these interpolation measures to representation theory.
We demonstrate  this relation by  calculation, and leave open the problem of finding
a deeper conceptual explanation for its existence.}


\subsection{Related work}\label{sec:12}
 The representations $A_n^k$ have  appeared in the literature in numerous places.
In particular, a 1995 result of Getzler \cite[Corollary 3.10]{Getzler:1995} permits an identification of
$A_n^k$ as an $S_n$-module with the $k$th cohomology group of the moduli space $\sM_{0, n+1}$ of
the Riemann sphere with $n+1$ marked points, viewed as an $S_n$-module, holding one point  fixed. Getzler identifies this cohomology
with the $S^1$-equivariant cohomology of 
$\pconf_n(\CC)$, which is  the cohomology of $Y_n$ given in Theorem \ref{thm:splitting-coeffs2}. 
Some more recent occurrences of $A_n^k$ are  discussed  in Section \ref{sec:stability}.

In connection with Theorem \ref{thm:main-2}, in  1996 Gaiffi \cite{Gai:1996} further explained Lehrer's 
formula $\theta= 2 \,\Ind_{\langle \tau \rangle}^{S_n}(1)$
 by showing that
\[
    H^{\bullet}( M(\sA_{n-1}), \CC) \simeq H^{\bullet}( M(d\sA_{n-1}), \CC) \otimes \big(\CC \oplus \tfrac{\CC[\varepsilon]}{ \varepsilon^2}\big),
\]
as $S_n$-modules, where $d\sA_{n-1}$ is obtained by a deconing construction, 
while the class $\varepsilon$ has degree $1$ and carries the
trivial $S_n$-action. (His space $M(\sA_{n-1})$ lies in $\CC^{n-1}$ and is obtained by restricting the braid arrangement on $\CC^n$ 
to the hyperplane $x_1+ x_2 + \cdots + x_n=0$ in $\CC^{n}$, and the deconed configuration space $M(d\sA_{n-1}) \subset \CC^{n-2}$.)
On comparison with our direct sum decomposition we have $H^k(d \sA_{n-1}, \CC) \simeq A_n^k$ as $S_n$-modules,
showing  that the deconed space $d\sA_{n-1}$ has an isomorphic cohomology ring as the complex manifold $Y_n$  with
an appropriate $S_n$-module structure.
Gaiffi and also Mathieu \cite{Mathieu:1996}
 showed there is a ``hidden''  $S_{n+1}$-action on this cohomology ring.
For more recent developments on the ``hidden'' action see Callegaro and Gaiffi \cite{CallegaroG:2015}.


\subsection{Plan of the Paper }\label{subsec: plan}
In Section \ref{sec:splitting-measure} we recall properties of the $z$-splitting measures from \cite{Lagarias-W:2015}.
In Section \ref{sec:rep-interpretation} we use the twisted Grothendieck-Lefschetz formula to relate  the coefficients of cycle polynomials  
to the characters of the $S_n$-representations $H^k(P_n,\QQ)$. 
In Section \ref{sec:braid-group} we discuss the cohomology $H^k(P_n, \QQ)$ of the pure braid group $P_n$,
and derive an exact sequence leading to the construction of the $S_n$-representations $A_n^k$.
In Section \ref{sec:splitting-characters} we express
 the splitting measure coefficients $\alpha_n^k(C_\lambda)$ in terms of the character $\chi_n^k$ of the representation $A_n^k$.
In Section \ref{sec:stability}  we discuss representation stability and connect the $S_n$-representations  $A_n^k$ with others in the literature.


\subsection{Notation} \label{subsec:notation}
\begin{enumerate}[leftmargin=*]
	\item $q=p^f$  denotes a prime power.
    
	\item The set of monic, degree $n$, square-free polynomials in $\FF_q[x]$ is denoted $\conf_n(\FF_q)$. 
    
    \item We write partitions either as $\lambda =\big[\lambda_1, \lambda_2, \cdots, \lambda_{\ell}\big]$, 
        with parts $\lambda_1 \ge \lambda_2 \ge \cdots$      
        eventually $0$, or as $\lambda = (1^{m_1} 2^{m_2}\cdots)$ where $m_j = m_j(\lambda)$ is the number of parts of $\lambda$ of size $j$.
 The length of $\lambda$ is $\ell(\lambda) = \max\{ r : \lambda_r \ge 1\}$,  the size of $\lambda$ is $|\lambda| = \sum_{i} \lambda_i = \sum_j{jm_j}$,
 and $\lambda_i$ is the $i$th largest part of $\lambda$. (Compare \cite{Macdonald:1995}.)

	\item Each partition $\lambda$ of $n$ corresponds to a conjugacy class 
	$C_{\lambda}$ of $S_n$ given by the common cycle structure of the elements in $C_\lambda$. 
	We let $Z_\lambda$ denote the centralizer of $C_{\lambda}$ in $S_n$. The size of the centralizer and conjugacy class are 
	\[
		z_\lambda := |Z_{\lambda}| = \prod_{j \geq 1}{j^{m_j(\lambda)}m_j(\lambda)!} \hspace{2em} c_\lambda:= |C_{\lambda}|= \frac{n!}{z_\lambda}
	\]
	respectively. Note that $c_\lambda z_\lambda = n!$.
	
	\item
	Following  Stanley \cite{Stanley:1997}, we let 
	$\part(n)$ denote the set of partitions of $n$ and $\part = \bigcup_n{\part(n)}$ the set of all partitions.  However 
	in  Section \ref{sec:stability}, we let $\Pi_n$ denote the set of partitions of $n$, partially ordered by refinement.
	
\end{enumerate}

\section{Splitting Measures}\label{sec:splitting-measure}
 
We review the splitting measures introduced in \cite{Lagarias-W:2015}, summarize their properties, and introduce the normalized splitting measures.

\subsection{Necklace polynomials and cycle polynomials}\label{subsec:necklace}


\begin{defi}\label{de21}
{\em 
For $j\geq 1$, the \emph{$j$th necklace polynomial} $M_j(z)\in \frac{1}{j}\ZZ[z]$ is
\[
    M_j(z) := \frac{1}{j}\sum_{d\mid j}{\mu(d) z^{j/d}},
\]
where $\mu(d)$ is the M\"{o}bius function.}
\end{defi}

Moreau \cite{Moreau:1872} noted in 1872  
that  for  all integers $m \ge 1$, $M_j(m)$ is the number of distinct
necklaces having $j$ beads drawn from a set of $m$ colors, up to cyclic permutation.
This fact motivated  Metropolis and Rota \cite{Metropolis:1983} to name them {\em necklace polynomials.}
Relevant to the present paper, $M_j(q)$ is the number of monic, degree $j$, irreducible polynomials in
 $\FF_q[X]$ \cite[Prop. 2.1]{Rosen:2002}. 
The factorization type of a polynomial $f\in \conf_n(\FF_q)$
is the partition formed by the degrees of its irreducible factors, which we write $[f]$.

\begin{defi}\label{def: cycle}
{\em 
Given a partition $\lambda$ of $n$, the \emph{cycle polynomial} $N_\lambda(z) \in \frac{1}{z_\lambda}\ZZ[z]$ is
\[
    N_\lambda(z) := \prod_{j\geq 1}{\binom{M_j(z)}{m_j(\lambda)}},
\]
where $\binom{\alpha}{m}$ is the usual extension of a binomial coefficient,
\[
	\binom{\alpha}{m} := \frac{1}{m!}\prod_{k=0}^{m-1}{(\alpha - k)}.
\]
}
\end{defi}

The cycle polynomial $N_\lambda(z)$ has degree $n = |\lambda|$ and is integer valued for $z\in\ZZ$.
The number of $f\in \conf_n(\FF_q)$ with $[f] = \lambda$ is $N_\lambda(q)$ (see \cite[Sect. 4]{Lagarias-W:2015}.)

\subsection{$z$-splitting measures}\label{subsec:splitting}

If $\lambda$ a partition of $n$, then the probability of a uniformly chosen $f\in \conf_n(\FF_q)$ having factorization type $\lambda$ is
\[
	\mathrm{Prob}\{f\in \conf_n(\FF_q) : [f] = \lambda\} = \frac{N_\lambda(q)}{|\conf_n(\FF_q)|}.
\]
When $n = 1$, $|\conf_n(\FF_q)| = q$ and for $n\geq 2$ we have $|\conf_n(\FF_q)| = q^n - q^{n-1}$.
(See \cite[Prop. 2.3]{Rosen:2002} for a proof via generating functions. A proof due to Zieve appears in \cite[Lem. 4.1]{Weiss:2013}.) Hence, the probability is a rational function in $q$. 
Replacing $q$ by a complex-valued parameter $z$ yields the $z$-splitting measure.

\begin{defi}\label{de23}
For $n \ge 2$ the  \emph{$z$-splitting measure} $\nu_{n,z}^*(C_{\lambda}) \in \QQ(z)$
 is given by
\[
	\nu_{n,z}^{\ast}(C_{\lambda}) := \frac{N_\lambda(z)}{z^n - z^{n-1}}.
\]
\end{defi}

\begin{prop}\label{laurent-poly}
For each partition $\lambda$ of $n \ge 1$, the rational function
$\nu_{n, z}^{\ast}(C_{\lambda})$ is a polynomial in $\frac{1}{z}$
of degree at most $n-1$. Thus it may be written as
\[
	\nu_{n, z}^{\ast}(C_{\lambda}) = \sum_{k=0}^{n-1} \alpha_n^k(C_{\lambda})\big(\tfrac{1}{z}\big)^k.
\]
 The function $\nu_{1, z}^{\ast}(C_{1}) = 1$ is independent of $z$.
\end{prop}

\begin{proof} The case $n=1$ is clear.
For $n\geq 2$ we have $N_\lambda(1) = 0$ by  \cite[Lemma 2.5]{Lagarias:2016}, whence $\frac{N_\lambda(z)}{z-1}$
is a polynomial of degree at most $n-1$ in $z$. Therefore,
\[
	\nu_{n, z}^{\ast}(C_\lambda) = \frac{N_\lambda(z)}{z^n - z^{n-1}} = \frac{1}{z^{n-1}}\left(\frac{N_\lambda(z)}{z-1}\right)
\]
is a  polynomial in $\frac{1}{z}$ of degree at most $n - 1$. 
\end{proof}

 For $n \ge 2$ the Laurent polynomial $\nu_{n, z}^{\ast}(C_{\lambda})$ is of degree at most $n-2$ since
   $z \mid N_{\lambda}(z)$ (\cite[Lemma 4.3]{Lagarias-W:2015}); that is,  $\alpha_n^{n-1}(C_{\lambda}) =0$.
Tables \ref{table11} and \ref{table12}   give 
 $\nu^*_{n,z}(C_{\lambda})$, exhibiting the splitting measure coefficients $\alpha_n^k(C_{\lambda})$
 for $n=4$ and $n=5$.
\begin{table}[h]
\renewcommand{\arraystretch}{1.3}
\begin{center}
	\begin{tabular}{| l |c |c| l | l }
	\hline
	$\lambda$ & $|C_\lambda|$ & $z_\lambda$ & $\nu^*_{4,z}(C_{\lambda})$   \\ \hline
	$[1,1,1,1]$ & $1$ & $24$& $\frac{1}{24}\big(1 -\frac{5}{z} + \frac{6}{z^2}\big)$ \\
	$[2,1,1]$ &  $6$ &  $4$ &$\frac{1}{4}\big(1 - \frac{1}{z}\big)$ \\
	$[2,2]$ & $3$ & $8$ &$\frac{1}{8}\big(1 - \frac{1}{z} - \frac{2}{z^2}\big)$ \\
	$[3,1]$ & $8$ & $3$ &$\frac{1}{3}\big(1 + \frac{1}{z}\big)$ \\
	$[4]$ &  $6$ & $4$ &$\frac{1}{4}\big(1 + \frac{1}{z}\big)$ \\
	\hline
	\end{tabular}
	\medskip
		\caption{Values of the $z$-splitting measures $\nu_{4, z}^{\ast}(C_\lambda)$  on partitions $\lambda$ of $n=4$.}
	\label{table11}
\end{center}
\end{table}

\begin{table}[h]
\renewcommand{\arraystretch}{1.3}
\begin{center}
	\begin{tabular}{| l |c |c| l | l |}
	\hline
	$\lambda$ &  $|C_\lambda|$ & $z_\lambda$ & $\nu^*_{5,z}(C_\lambda)$ \\ \hline
	$[1,1,1,1,1]$ &$1$ & $120$ &$\frac{1}{120}\big(1 - \frac{9}{z} + \frac{26}{z^2} - \frac{24}{z^3}\big)$ \\
	$[2,1,1,1]$ & $10$ &$12$ &$\frac{1}{12}\big(1 - \frac{3}{z} + \frac{2}{z^2}\big)$ \\
	$[2,2,1]$ & $15$ &$8$ &$\frac{1}{8}\big(1 - \frac{1}{z} - \frac{2}{z^2}\big)$ \\
	$[3,1,1]$ & $20$ &$6$ &$\frac{1}{6}\big(1 - \frac{1}{z^2}\big)$ \\
	$[3,2]$ & $20$ &$6$ &$\frac{1}{6}\big(1 - \frac{1}{z^2}\big)$ \\
	$[4,1]$ & $30$ &$4$ &$\frac{1}{4}\big(1 + \frac{1}{z}\big)$ \\
	$[5]$ & $24$ &$5$ &$\frac{1}{5}\big(1 + \frac{1}{z} + \frac{1}{z^2} + \frac{1}{z^3}\big)$ \\
	\hline
	\end{tabular}
	\medskip
	\caption{Values of the $z$-splitting measures $\nu_{5, z}^{\ast}(C_\lambda)$ on partitions $\lambda$ of $n=5$.}
	\label{table12}
\end{center}
\end{table}

\newpage
\section{Interpretation of Cycle Polynomial Coefficients}\label{sec:rep-interpretation}

In Section \ref{subsec:necklace} we defined
the cycle polynomials $N_\lambda(z) \in \frac{1}{z_\lambda}\ZZ[z]$  for each partition $\lambda$ of  $n$.
In this section  we express the coefficients of $N_\lambda(z)$ as a function of $\lambda$
in terms of characters $h_n^k$ of the cohomology of the pure braid group $P_n$ viewed as an $S_n$-representation.
We establish this connection using  the twisted Grothendieck-Lefschetz formula of Church, Ellenberg, and Farb \cite{CEF:2014}.
 Using explicit formulas for the cycle polynomials we obtain constraints on the support of $h_n^k$,
 and we compute $h_n^k(\lambda)$ for varying $n$ in several examples.
 
\subsection{Cohomology of the pure braid group}\label{subsec:braid-group}

Given a set $X$ of $n$ distinct points in $3$-dimensional
affine space, the \emph{braid group} $B_n$ consists of homotopy classes of simple, non-intersecting paths beginning and terminating in $X$, with concatenation as the group operation. Each element of $B_n$ determines a permutation of $X$, giving a short exact sequence of groups
\[
	0 \rightarrow P_n \rightarrow B_n \xrightarrow{\pi} S_n \rightarrow 0.
\]
Then  $P_n := \ker \pi$
is called the \emph{pure braid group}. $P_n$ consists of homotopy classes of simple, non-intersecting \emph{loops} based in $X$. The action of $S_n$ on $X$ induces an action on $P_n$ by permuting the loops. Thus, for each $k$, the $k$th group cohomology $H^k(P_n,\QQ)$ carries  an $S_n$-representation whose character we denote by $h_n^k$.
 
 \subsection{Twisted Grothendieck-Lefschetz formula}\label{subsec:twist}

A \emph{character polynomial} is a polynomial $P(x) \in \QQ[x_j : j\geq 1]$.
Character polynomials induce functions $P: \part \rightarrow \QQ$ by
\[
    P(\lambda) := P\big(m_1(\lambda), m_2(\lambda), \ldots\big),
\]
noting that $m_i(\lambda) = 0$ for all but finitely many $i$. 
For $f \in \conf_n(\FF_q)$ we let $P(f) := P([f])$.
Given two $\QQ$-valued functions $F$ and $G$ defined on $S_n$ let 
\[
	\langle F, G \rangle := \frac{1}{n!}\sum_{g \in S_n}{F(g)G(g)}.
\]
The following Theorem is due to Church, Ellenberg, and Farb \cite[Prop. 4.1]{CEF:2014}.

\begin{thm}[Twisted Grothendieck-Lefschetz formula for $\pconf_n$]
\label{GrothLef}
Given a prime power $q$, an integer $n\geq 1$, and a character polynomial $P$, we have
\begin{equation}\label{twistEq}
    \sum_{f\in \conf_n(\FF_q)}{P(f)} = \sum_{k=0}^n{(-1)^k \big\langle P, h_n^k\big\rangle\, q^{n-k}},
\end{equation}
where $h_n^k$ is the character of
the cohomology of the pure braid group $H^k(P_n, \QQ)$.
\end{thm}

The classic {Lefschetz trace formula} counts the fixed points of an endomorphism $f$ on a compact manifold $M$ by the
 trace of the induced map on the singular cohomology of $M$. 
One may interpret the $\overline{\FF}_q$ points on an algebraic variety $V$ defined over $\FF_q$ as the fixed points of the 
\emph{geometric Frobenius endomorphism} of $V$. 
Using the machinery of $\ell$-adic \'{e}tale cohomology, Grothendieck \cite{Gro:1963} generalized Lefschetz's formula to
 count the number of points in $V(\FF_q)$
 by the trace of
Frobenius on the \'{e}tale cohomology of $V$. 
For nice varieties $V$ defined over $\ZZ$, there are comparison theorems relating the \'{e}tale cohomology of $V(\overline{\FF}_q)$ to 
the singular cohomology of $V(\CC)$. 
This connects the topology of a complex manifold to point counts of a variety over a finite field.  For hyperplane complements the connection was
made in 1992 by Lehrer \cite{Lehrer:1992}, and for equivariant actions of a finite group on varieties the 
equivariant Poincar\'{e} polynomials were determined by Kisin and Lehrer \cite{KisinL:2002} in 2002.

Church, Ellenberg, and Farb \cite{CEF:2014} build upon Grothendieck's extension of the Lefschetz formula to relate point counts on natural subsets of $\conf_n(\FF_q)$ to the singular cohomology of the covering space $\pconf_n(\CC) \rightarrow \conf_n(\CC)$. 
 $\pconf_n(\CC)$ is the space of $n$ distinct, labelled points in $\CC$. The space $\pconf_n(\CC)$ has fundamental group $P_n$, the pure braid group, 
 and is a $K(\pi,1)$ for this group. Hence, the singular cohomology of $\pconf_n(\CC)$ is the same as the group cohomology of $P_n$. 
 This fact yields the  connection between $\conf_n(\FF_q)$ on the left hand side of \eqref{twistEq} and the character of the pure braid group cohomology.
 
%
%
\subsection {Cycle polynomials and pure braid group cohomology}\label{subsec:cycle-braid}

We express the coefficients of the cycle polynomials $N_\lambda(z)$ in terms of the characters $h_n^k$ as an application of Theorem \ref{GrothLef}.
Theorem \ref{cycle-coeffs} is equivalent to Lehrer's \cite[Theorem 5.5]{Lehrer:1987} by comparing numerators and making a slight change of variables.

\begin{thm}\label{cycle-coeffs}
Let $\lambda$ be a partition of $n$,
then
\[
    N_\lambda(z) = \frac{1}{z_\lambda}\sum_{k=0}^n{(-1)^k h_n^k(\lambda) z^{n-k}},
\]
where $h_n^k$ is the character of the $S_n$-representation $H^k(P_n,\QQ)$.
\end{thm}

\begin{proof}
Define the character polynomial $1_\lambda(x) \in \QQ[x_j : j \geq 1]$ by
\[
	1_\lambda(x) = \prod_{j\geq 1}{\binom{x_j}{m_j(\lambda)}}.
\]
Observe that for a partition $\mu \in \part(n)$ we have
\[
	1_\lambda(\mu) = \begin{cases} 1 & \text{if }\mu = \lambda,\\ 0 & \text{otherwise.}\end{cases}
\]
Therefore,
\[
    N_\lambda(q) = \sum_{f\in \conf_n(\FF_q)}{1_\lambda(f)}.
\]
On the other hand, by Theorem \ref{GrothLef} we have
\[
    \sum_{f\in \conf_n(\FF_q)}{1_\lambda(f)} = \sum_{k=0}^n{(-1)^k \big\langle 1_\lambda, h_n^k\big\rangle q^{n-k} }.
\]
If $g \in S_n$, let $[g] \in \part(n)$ be the partition given by the cycle lengths of $g$. Thus,
\[
    \big\langle 1_\lambda, h_n^k\big\rangle = \frac{1}{n!}\sum_{g \in S_n}{1_\lambda(g)h_n^k(g)} = 
    \frac{1}{n!}\sum_{\substack{g \in S_n\\ [g]=\lambda}}{h_n^k(g)} = \frac{c_\lambda}{n!}h_n^k(\lambda) = 
    \frac{1}{z_\lambda}h_n^k(\lambda).
\]
Therefore the identity
\[
	N_\lambda(q) = \frac{1}{z_\lambda}\sum_{k=0}^n{(-1)^k h_n^k(\lambda) q^{n-k}}
\]
holds for all prime powers $q$, giving the identity as polynomials in $\QQ[z]$.
\end{proof}
\begin{remark}
A recent result of Chen \cite[Theorem 1]{Chen:2016} also yields the identity in Theorem \ref{cycle-coeffs} by specializing at $t=0$. 
\end{remark}

One can explicitly  compute $h_n^k(\lambda)$ using Theorem \ref{cycle-coeffs}
by expanding the formula \eqref{eq:11} for $N_\lambda(z)$ and comparing coefficients. 
Lehrer \cite{Lehrer:1987} derives several corollaries this way. 
Here we give further  examples intended to explore possible connections with number theory.
We  obtain  restrictions  on the support of $h_n^k$ in Proposition \ref{prop:support}.  Then we
compute  values of $h_n^k(\lambda)$  in Sections \ref{subsec:application-1} and \ref{subsec:application-2}.
For any fixed $k$, the $h_n^k$ are given by character polynomials, while $h_n^{n-k}$ for $k < 2n/3$ exhibit interesting arithmetic structure.

%
%
\subsection{Support restrictions on  characters  $h_n^k$ }\label{subsec:application-0}

The character $h_n^k$ is supported on partitions with at least one small part, while  $h_n^{n-k}$ is supported on partitions
having at most $k$ different parts. The latter are  {\em multi-rectangular Young diagrams} having at most $k$ steps,
using the terminology of Do\l ega et al. \cite[Sect. 1.7]{DolegaFS:2010} and \'{S}niady \cite{Sniady:2014}.

\begin{prop}\label{prop:support}
Let $0 \le k \le n$  and $h_n^k$ be the character of the $S_n$-representation $H^k(P_n, \QQ)$, then

(1) $h_n^k$ is supported on partitions  having at least one part of size
at most $2k$. The value $h_n^k(\lambda)$ is determined by $m_j(\lambda)$ for $1 \le j \leq 2k$.

(2) $h_n^{n-k}$  is supported on multi-rectangular partitions $\lambda$ having at most $k$ distinct values of $j$
with  $m_j(\lambda)>0$.
\end{prop}

\begin{proof}
(1) Theorem \ref{cycle-coeffs} implies $h_n^k(\lambda)$ is nonzero iff
the coefficient of $z^{n-k}$ in $N_{\lambda}(z)$ is nonzero.
The degree of $M_j(z) - \frac{1}{j}z^j$ is at most $\lfloor j/2 \rfloor$. Hence if $j > 2k$, then the coefficient of $z^{n-k}$ in $\Big( {M_j(z) \atop m_j(\lambda)} \Big)$ is zero.
Thus the only $j$ contributing to the coefficient of $z^{n-k}$ in $N_{\lambda}(z)$ in \eqref{eq:11} are those with $1 \le j\leq 2k$.

(2) Theorem \ref{cycle-coeffs} implies $h_n^{n-k}(\lambda)$ is nonzero iff the coefficient of $z^k$ in $N_\lambda(z)$ is nonzero. If $m_j(\lambda) > 0$, then $z$ divides $\Big( {M_j(z) \atop m_j(\lambda)} \Big)$. Hence if $m_j(\lambda) > 0$ for more than $k$ values of $j$, then $h_n^{n-k}(\lambda) = 0$.
\end{proof}

\begin{remark}
{Property (1) is a manifestation of representation stability of $h_n^k$, which says that for fixed $k$ and all sufficiently large $n$,
the values of  $h_n^k(\lambda)$ are described by
a character polynomial in $\lambda$. A {\em character polynomial} for a partition $\lambda=(1^{m_1}2^{m_2} \cdots n^{m_n})$ is a polynomial in 
the variables $m_j$, see Example  \ref{h1-h2-example}. 
 Farb \cite{Farb:ICM2014} raised the problem  of explicitly determining  such  character polynomials.}
Proposition \ref{prop:support} bounds which variables $m_j$ may occur in the character polynomial for $h_n^k$.
A known sharp representation stability property  of $h_n^k$ is that it equals such a character polynomial for all $n \ge 3k+1$,
as shown in \cite[Theorem 1.1]{HerRein:2015}, taking dimension $d=2$. 
\end{remark}

%
%
\subsection{Character values $h_n^k(\lambda)$ for fixed $\lambda$ and varying $k$}\label{subsec:application-1}
We  give special cases of  explicit determinations for  $h_n^k(\lambda)$ for various fixed $\lambda$ and varying $k$ by 
directly expanding the cycle polynomial $N_\lambda(z)$.

\begin{example}[Dimensions of cohomology]
The dimension of $H^k(P_n,\QQ)$ is the value of $h_n^k$ at the identity element, corresponding to the partition $(1^n)$. Since $M_1(z) = z$ and the centralizer of the identity in $S_n$ has order $z_{(1^n)} = n!$, we have
\[
	N_{(1^n)}(z) = \binom{z}{n} = \frac{1}{n!}\prod_{i=0}^{n-1}(z - i) = \frac{1}{n!}\sum_{k=0}^n{(-1)^k{n \brack n-k}z^{n-k}},
\]
where ${n \brack n-k}$ is an  
{\em unsigned Stirling number of the first kind}. Theorem \ref{cycle-coeffs} says
\[
	N_{(1^n)}(z) = \frac{1}{n!}\sum_{k=0}^n{(-1)^k h_n^k\big((1^n)\big) z^{n-k}}.
\]
Comparing coefficients recovers the well-known formula due to  Arnol'd \cite{Arn:1969}
for the dimension of the pure braid group cohomology:
\[
	\mathrm{dim}\,H^k(P_n,\QQ) = h_n^k\big((1^n)\big) = {n \brack n-k}.
\]
These values are given in Table \ref{betti}.
\end{example}

\begin{table}[h]
\renewcommand{\arraystretch}{0.8}
\begin{center}
	\begin{tabular}{| c | r |r | r| r | r |r |r|r|r|r|r|}
	\hline
	$n$ $\backslash$ $k$ &  $0$ & $1$ & $2$ & $3$ & $4$ & $5$ & $6$  & $7$ & $8$  \\ \hline
	$1$ & $1$ &$0$ & $0$ & $0$ & $0$ & $0$ & $0$ & $0$ & $0$  \\
        $2$ &$1$ & $1$ & $0$ & $0$ & $0$ & $0$ & $0$ & $0$ & $0$  \\
        $3$ & $1$ &$3$ & $2$ & $0$ & $0$ & $0$ & $0$ & $0$ & $0$  \\
	$4$ & $1$ &$6$ &$11$ & $6$ & $0$ & $0$ & $0$ & $0$ & $0$  \\
	$5$ & $1$ &$10$ &$35$ & $50$ & $24$ & $0$ & $0$    & $0$ & $0$  \\
	$6$ & $1$ &$15$ &$85$ & $225$ & $274$ & $120$ & $0$        & $0$ & $0$ \\
	$7$ & $1$ &$21$ &$175$ & $735$ & $1624$ & $1764$  & $720$ & $0$ & $0$ \\
        $8$ & $1$ &$28$ &$322$ & $1960$ & $6769$ & $13132$   & $13068$ & $5040$ & $0$  \\
        $9$ & $1$ &$36$ &$546$ & $4536$ & $22449$ & $67284$ & $118124$   & $109584$ & $40320$\\ 
	\hline
	\end{tabular}
	\medskip
	\caption{Betti numbers of pure braid group cohomology $H^k(P_n, \QQ)$.}
	\label{betti}
\end{center}
\end{table}

\begin{example}
The partition $\lambda = [n]$ corresponds to an $n$-cycle in $S_n$. The centralizer of an $n$-cycle has order $z_{[n]} = n$ and
\begin{equation}\label{first-exp}
	N_{[n]}(z) = \binom{M_n(z)}{1} = M_n(z) = \frac{1}{n}\sum_{d\mid n}{\mu(d) z^{n/d}}.
\end{equation}
Theorem \ref{cycle-coeffs} gives us
\begin{equation}\label{second-exp}
	N_{[n]}(z) = \frac{1}{n}\sum_{k=0}^n{(-1)^kh_n^k\big([n]\big)z^{n-k}}.
\end{equation}
 Comparing coefficients, we find that
   \begin{equation*}\label{eqn:304}
   h_n^{n-k}\big([n]\big) = 
 \left\{
\begin{array}{cl}
(-1)^{n-k} \mu(\frac{n}{k})  & \mbox{if}~ k \mid n, \\
~~~ \\
0 & \mbox{if} ~~k\nmid n.
\end{array}
\right.
 \end{equation*}

 \end{example}

\subsection{Character values $h_n^k(\lambda)$ for fixed $k$ and varying $\lambda$}\label{subsec:application-2}
We now compute $h_n^k(\lambda)$ for fixed $k$ and varying $\lambda$.

\begin{example}[Computing $h_n^0$ and $h_n^n$]\label{h0 example}
 The cases $k=0$ and $n$ are both constant: $h_n^0 = 1$ and $h_n^n = 0$. 
The leading coefficient of $N_\lambda(z)$ is $1/z_\lambda$, hence Theorem \ref{cycle-coeffs} tells us $h_n^0(\lambda) = 1$ for all $\lambda$. For $j\geq 1$, we have $z \mid M_j(z)$, from which it follows that $z \mid N_\lambda(z)$ for all partitions $\lambda$ of $n \geq 1$. In other words, for all $m_j \geq 1$
\[
    \frac{1}{z_\lambda}(-1)^n h_n^n(\lambda) = N_\lambda(0) = 0.
\]
Thus $h_n^n(\lambda) = 0$ for all $\lambda$, and $H^n(P_n, \QQ) =0.$
\end{example}

\begin{example}[Computing $h_n^1$ and $h_n^2$] \label{h1-h2-example}
 Taking $\lambda= ( 1^{m_1} 2^{m_2} \cdots)$,
a careful analysis of the $z^{n-1}$ and $z^{n-2}$ coefficients in $N_\lambda(z)$ and Theorem \ref{cycle-coeffs} yields the following formulas
\begin{align*}
h_n^1(\lambda) &= \binom{m_1}{2} + \binom{m_2}{1}\\
h_n^2(\lambda) &= 2\binom{m_1}{3} + 3\binom{m_1}{4} +\binom{m_1}{2}\binom{m_2}{1}-\binom{m_2}{2} - \binom{m_3}{1} - \binom{m_4}{1},
\end{align*}
where $m_j = m_j(\lambda)$. 
These formulas represent $h_n^1$ and $h_n^2$ as character polynomials, and they appear
in \cite[Lemma 4.8]{CEF:2014}.
Note that  $h_n^1(\lambda) = h_n^2(\lambda) = 0$ for partitions $\lambda$ having all parts larger than 2 and 4 respectively,
 illustrating  Proposition \ref{prop:support}  (1).  
\end{example}

\begin{example}[Computing $h_n^{n-1}$]
The $z$ coefficient of $N_\lambda(z)$ determines the value of $h_n^{n-1}(\lambda)$. Since each $j$ with $m_j(\lambda)>0$ contributes a factor of $z$ to $N_\lambda(z)$, $h_n^{n-1}$ is supported on partitions of the form $\lambda = (j^{m})$. Note that the $z$ coefficient of the necklace polynomial $M_j(z)$ is $\mu(j)/j$. Let $\lambda = (j^{m})$, then the $z$ coefficient of
\[
    N_\lambda(z) = \binom{M_j(z)}{m} = \frac{M_j(z)(M_j(z) - 1)\cdots(M_j(z) - m + 1)}{m!}
\]
is $(-1)^{m -1}\frac{\mu(j)}{j m}$. Since $z_\lambda = j^{m}m!$, we conclude
\begin{equation*}\label{n-1}
    h_n^{n-1}(\lambda) = 
    \left\{
\begin{array}{cl}
(-1)^{m - n}\mu(j) j^{m - 1} (m - 1)!  & \mbox{if}~ \lambda= (j^{m}), \\
~~~ \\
0 & \mbox{otherwise}.
\end{array}
\right.
 \end{equation*}
 By \cite[Corollary $(5.5)^{' {} '}$]{Lehrer:1987} $h_n^{n-1} = \Sign_n \otimes \Ind_{c_{n}}^{S_n}(\zeta_n),$
 where $c_n$ is a cyclic group of order $n$ and $\zeta_n$ is a faithful character on it, 
 noted earlier by Stanley \cite{Stanley:1982}.
\end{example}

\begin{example}[Computing $h_n^{n-2}$]
The $z^2$ coefficient of $N_\lambda(z)$ determines $h_n^{n-2}(\lambda)$.
Proposition \ref{prop:support} (2)
tells us that $h_n^{n-2}(\lambda) = 0$ when $m_j(\lambda) > 0$ for at least three $j$. We
treat the two remaining cases $\lambda = (i^{m_i} j^{m_j})$ and $\lambda = ( j^{m})$ in turn.
		If $\lambda = (i^{m_i}j^{m_j})$, then the  $z$ coefficient of $\Big( {M_i(z) \atop m_i} \Big)$ is $(-1)^{m_i -1}\frac{\mu(i)}{i m_i}$, 
		and similarly for $\Big( {M_j(z) \atop m_j} \Big)$. 
		We have $z_\lambda = (i^{m_i} m_i!)(j^{m_j}m_j!)$. Thus, by Theorem \ref{cycle-coeffs}
		\begin{align*}
			h_n^{n-2}\big( (i^{m_1} j^{m_j}) \big)&= (-1)^{m_i + m_j - n} z_\lambda \frac{\mu(i)\mu(j)}{(i m_i )(jm_j)}\\
			 &= (-1)^{m_i + m_j - n} \big(\mu(i)i^{m_i -1}(m_i -1)!\big)\big(\mu(j) j^{m_j-1} (m_j-1)!\big).
		\end{align*}

		If $\lambda = (j^{m})$, then
		the $z^2$ coefficient of $N_\lambda(z)$ receives a contribution of $(-1)^{m-1}\frac{\mu(j/2)}{j m}$ from the quadratic term of $M_j(z)$ if $j$ is even. 
		The $z$ coefficient of 
		$\Big({M_j(z) \atop m_j}\Big)/M_j(z)$ is 
		\[
			\frac{\mu(j)}{j m!}\left( \sum_{i=1}^{m-1}{\frac{(-1)^{m-2}(m-1)!}{i}} \right) = (-1)^{m} \frac{\mu(j)}{j m} H_{m -1},
		\]
		where  $H_{m-1}=\sum_{i=1}^{m-1} \frac{1}{i}$ denotes the {\em $(m-1)$th harmonic number}. The $z$ coefficient of $M_j(z)$ is $\frac{\mu(j)}{j}$.
		 Using the convention that the M\"{o}bius function  $\mu(\alpha) $ vanishes at non-integral $\alpha$, 
		  we arrive at the following expression for $h_n^{n-2}(\lambda)$:
		\begin{align*}
			h_n^{n-2}\big( (j^{m}) \big) &= z_\lambda (-1)^{m - n}\frac{\big(\mu(j)^2 H_{m-1} - \mu(\frac{j}{2})\big)}{jm}\\
			&= (-1)^{m - n}\big(\mu(j)^2 H_{m-1} - \mu(\tfrac{j}{2})\big)j^{m - 1} (m -1)!.
		\end{align*}
\end{example}

%
%

 \section{Submodules  $A_n^k$ of Pure Braid Group Cohomology}\label{sec:braid-group}
 
 Starting from Arnol'd's presentation for the $S_n$-algebra $H^\bullet(P_n,\QQ)$ we obtain a decomposition $H^k (P_n, \QQ) = A_n^{k-1} \oplus A_n^k$ of $S_n$-modules.
 The characters of the sequence $A_n^k$ of $S_n$-modules determine the splitting measure coefficients $\alpha_n^k(C_\lambda)$.
 In Section \ref{sec:43a} we interpret $A_n^\bullet$ as the cohomology of $\pconf_n(\CC)/\CC^\times$, where $\CC^\times$ acts freely on $\pconf_n(\CC)$ by scaling coordinates. 
 
%
%
\subsection{Presentation of pure braid group cohomology ring}\label{sec:40}

Arnol'd \cite{Arn:1969} gave the following presentation of the cohomology ring $H^{\bullet}(P_n,\QQ)$ of the pure braid group $P_n$ as an $S_n$-algebra.
\begin{thm}[Arnol'd] \label{arnold}
There is an isomorphism of graded $S_n$-algebras
\[
    H^{\bullet}(P_n,\QQ) \cong \Lambda^{\bullet}[\omega_{i,j}]/\langle R_{i,j,k} \rangle,
\]
where $1 \leq i,j,k \leq n$ are distinct, $\omega_{i,j} = \omega_{j,i}$ have degree $1$, and
\[
    R_{i,j,k} = \omega_{i,j}\wedge\omega_{j,k} + \omega_{j,k}\wedge\omega_{k,i} + \omega_{k,i}\wedge\omega_{i,j}.
\]
An element $g \in S_n$ acts on $\omega_{i,j}$ by $g\cdot \omega_{i,j} = \omega_{g(i),g(j)}$.
\end{thm}

In what follows, we identify $H^{\bullet}(P_n,\QQ)$ with this presentation as a quotient of an exterior algebra. 
The ring $\Lambda^{\bullet}[\omega_{i,j}]/\langle R_{i,j,k} \rangle$ is an example of an \emph{Orlik-Solomon algebra},
 which arise as cohomology rings of complements of hyperplane arrangements 
 (see  Orlik and Solomon \cite{OrlikS:1980}, Dimca and Yuzvinsky \cite{DY:2010}, and Yuzvinsky \cite{Yuzvinsky:2001}.)

%
%
\subsection{$S_n$-modules $A_n^k$ inside  braid group cohomology} \label{sec:N41}
Let $\tau = \sum_{1\leq i< j \leq n}{\omega_{i,j}} \in H^1(P_n, \QQ)$. 
The element  $\tau$ generates a trivial $S_n$-subrepresentation of $H^1(P_n,\QQ)$. 
We define maps $d^k: H^k(P_n,\QQ) \rightarrow H^{k+1}(P_n,\QQ)$ for each $k$ by $\nu \mapsto \nu \wedge \tau$. This map is
 linear and $S_n$-equivariant, since 
\[
    g \cdot d^k(\nu) = g\cdot (\nu \wedge \tau) = (g\cdot \nu) \wedge (g \cdot \tau) = (g\cdot \nu) \wedge \tau = d^k(g\cdot \nu).
\]
From $d^{k+1}\circ d^k = 0$ we conclude that
\begin{equation*}
    0 \rightarrow H^0(P_n, \QQ) \xrightarrow{d^0} H^1(P_n, \QQ) \xrightarrow{d^1} \cdots \xrightarrow{d^{n-1}} H^n(P_n, \QQ) \xrightarrow{d^n} 0
\end{equation*}
is a chain complex of $S_n$-representations. It follows from the general theory of Orlik-Solomon algebras that the above sequence 
is exact \cite[Thm. 5.2]{DY:2010}. 
We include a proof in this case for completeness. 

\begin{prop}\label{prop:exact-lemma}
In the above notation,
\begin{equation} \label{Exact}
    0 \rightarrow H^0(P_n, \QQ) \xrightarrow{d^0} H^1(P_n, \QQ) \xrightarrow{d^1} \cdots \xrightarrow{d^{n-1}} H^n(P_n, \QQ) \xrightarrow{d^n} 0
\end{equation}
is an exact sequence of $S_n$-representations.  Set  $A_n^k := \im(d^k)
\subset H^{k+1}(P_n, \QQ)$.
Hence we have an isomorphism of $S_n$-representations for each $k$,
\begin{equation*}\label{k-split}
	H^k(P_n, \QQ) \cong A_n^{k-1} \oplus A_n^{k}.
\end{equation*}
\end{prop}

\begin{proof}
Arnol'd \cite[Cor. 3]{Arn:1969} describes an additive basis $\cB_k$ for $H^k(P_n, \QQ)$ comprised of all simple wedge products
\begin{equation*}
    \omega_{i_1, j_1} \wedge \cdots \wedge \omega_{i_k, j_k}\text{ such that } i_s < j_s\text{ for each $s$, and }j_1 < j_2 < \ldots < j_k.
\end{equation*}
Let
\[
	U_k = \{ \omega_{i_1,j_1} \wedge \cdots \wedge \omega_{i_k,j_k} \in \cB_k : (i_s, j_s) \neq (n-1, n)\},
\]
for $k >0$ and $U_0 = \{1\}$. Then set
\[
	\cC_k = U_k \cup \{\omega \wedge \tau : \omega \in U_{k-1}\}.
\]

\noindent{\bf Claim.} {\em $\cC_k$ is a basis of $H^k(P_n,\QQ)$.}\\

For example, we have
\[
    \cC_1 = \{\omega_{i,j} : (i,j) \neq (n-1, n)\} \cup \{\tau\},
\]
which is clearly a basis for $H^1(P_n,\QQ)$.

To prove the claim, since $|\cB_k| = |\cC_k|$, it suffices to show $\cC_k$ spans $H^k(P_n, \QQ)$. Note that
\[
	\cB_k = U_k \cup \{\omega \wedge \omega_{n-1,n} : \omega \in U_{k-1}\},
\]
further reducing the problem to expressing $\omega \wedge \omega_{n-1,n}$ as a linear combination of $\cC_k$ for each $\omega \in U_{k-1}$.
 Given $\omega = \omega_{i_1,j_1} \wedge \cdots \wedge \omega_{i_{k-1},j_{k-1}} \in U_{k-1}$, we use the relation
\[
	\omega_{i_s,j}\wedge \omega_{i,j} = \omega_{i_s,i}\wedge\omega_{i,j} - \omega_{i_s,i}\wedge\omega_{i_s,j}
\]
to express $\omega \wedge \omega_{i,j}$ in terms of elements of $U_k$ as follows:
\[
	\omega \wedge \omega_{i,j} = \begin{cases}
	\pm \omega_{i_1,j_1}\wedge \cdots \wedge \omega_{i_s,j_s}\wedge \omega_{i,j} \wedge \omega_{i_{s+1},j_{s+1}} \wedge \cdots \wedge \omega_{i_{k-1},j_{k-1}}\\
	  \quad\quad\quad \quad\quad\quad\quad\quad \quad \mbox{for}  \quad j_s < j < j_{s+1},\\
	\pm \omega_{i_1,j_1} \wedge \cdots \wedge (\omega_{i_s,i}\wedge \omega_{i,j} - \omega_{i_s,i}\wedge \omega_{i_s,j}) \wedge \cdots \wedge \omega_{i_{k-1},j_{k-1}}\\
	  \quad\quad\quad \quad\quad\quad\quad\quad \quad \mbox{for}  \quad j_s = j, i_s \neq i,\\
	0  \quad\quad \quad\quad\quad \quad\quad\quad\,\,\, \mbox{for} \quad (i_s, j_s) = (i,j). 
	\end{cases}
\]
The first and third cases are easily seen to belong in the span of $U_k$. Since $i_s, i < j$ and $j$ does not occur twice as a largest subscript in
 $\omega$, we see inductively that the second case also belongs in the span of $U_k$. Therefore, 
 $\omega \wedge \tau = \omega \wedge \omega_{n-1,n} + \nu$, where $\nu$ is in the span of $U_k$. Hence $\omega\wedge \omega_{n-1,n} = \omega \wedge \tau -\nu$ is in the span of $\cC_k$ and we conclude that $\cC_k$ is a basis, proving the claim.

We now show the sequence \eqref{Exact} is exact. Suppose $\nu \in \ker(d^k)$. Express $\nu$ in the basis $\cC_k$ as 
\[
    \nu = \sum_{\omega \in U_k}{a_\omega\, \omega} + \sum_{\omega\in U_{k-1}}{b_\omega\, \omega}\wedge \tau.
\]
Then
\[
    0 = d^k(\nu) = \nu \wedge \tau = \sum_{\omega \in U_k}{a_\omega \, \omega \wedge \tau}.
\]
Since $\omega \wedge \tau$ is an element of the basis $\cC_{k+1}$ for each $\omega \in U_k$, we have $a_\omega= 0$. Hence, $\nu = \mu \wedge \tau = d^{k-1}(\mu)$ where
\[
    \mu = \sum_{\omega \in U_{k-1}}{b_\omega \, \omega},
\]
so $\ker(d^k)= \im(d^{k-1}).$
\end{proof}

Recall from Section  \ref{subsec:application-1} that the 
 dimension of $H^k(P_n, \QQ)$ is given by an unsigned Stirling number of the first kind
\[
	\dim\big(H^k(P_n,\QQ)\big)= {n \brack n - k},
\]
where the unsigned Stirling numbers are determined by the identity
$\prod_{k=0}^{n-1}{(x + k)} = \sum_{k=0}^{n-1}{{n \brack k} x^k}.$
The exact sequence in Proposition \ref{prop:exact-lemma} shows the  dimension of $A_n^k$ is
\[
	\dim(A_n^k) = \sum_{j=0}^k {(-1)^j {n \brack n - k + j}}.
\]
Table \ref{table4} gives values of $\dim(A_n^k)$ for small $n$ and $k$; here $\dim (A_n^{n-1})=0$ for $n \ge 2$.

\begin{table}[h]
\renewcommand{\arraystretch}{0.9}
\begin{center}
	\begin{tabular}{| c | r |r | r| r | r |r |r|r|r|r|}
	\hline
	$n$ $\backslash$ $k$ &  $0$ & $1$ & $2$ & $3$ & $4$ & $5$ & $6$  & $7$ \\ 
	\hline
       $1$ & $1$ &$0$ & $0$ & $0$ & $0$ & $0$ & $0$ & $0$\\
        $2$ & $1$ &$0$ & $0$ & $0$ & $0$ & $0$ & $0$ & $0$\\
	$3$ & $1$ &$2$ &$0$ & $0$ & $0$ & $0$ & $0$ & $0$\\
	$4$ & $1$ &$5$ &$6$ & $0$ & $0$ & $0$ & $0$ & $0$\\
	$5$ & $1$ &$9$ &$26$ & $24$ & $0$ & $0$ & $0$  & $0$ \\
	$6$ & $1$ &$14$ &$71$ & $154$ & $120$ & $0$  & $0$ & $0$\\
	$7$ & $1$ &$20$ &$155$ & $580$ & $1044$ & $720$   & $0$ & $0$\\
        $8$ & $1$ &$27$ &$295$ & $1665$ & $5104$ & $8028$   & $5040$ & $0$\\
        $9$ & $1$ &$35$ &$511$ & $4025$ & $18424$ & $48860$    & $69264$ & $40320$\\
	\hline
	\end{tabular}
	\medskip
	\caption{$\dim(A_n^k)$}
	\label{table4}
\end{center}
\end{table}

%
%

 \subsection{$A_n^k$ as cohomology of a  complex manifold with an $S_n$-action}\label{sec:43a}
 
Recall from Section \ref{subsec:twist}
that the pure configuration space $\pconf_n(\CC)$ is defined by
\[
	\pconf_n(\CC) = \{(z_1, z_2, \ldots, z_n) \in \CC^n : z_i \neq z_j \text{ when } i\neq j\}.
\]
It is an open complex manifold, and 
the symmetric group $S_n$ acts on $\pconf_n(\CC)$ by permuting coordinates. There is also a free action of $\CC^\times$ on $\pconf_n(\CC)$ defined by
\[
	c\cdot(z_1,z_2,\ldots, z_n) = (c z_1, c z_2, \ldots, c z_n).
\]
This action commutes with the $S_n$-action, hence it induces an action of $S_n$ on the quotient complex manifold $\pconf_n(\CC)/\CC^\times$. 
Therefore $H^\bullet(\pconf_n(\CC)/\CC^\times,\QQ)$ is an $S_n$-algebra. We now relate the graded components 
$H^k(\pconf_n(\CC)/\CC^\times, \QQ)$ to the $S_n$-submodules $A_n^k$ of 
$H^k(\pconf_n(\CC),\QQ) = H^k(P_n,\QQ)$ constructed in Proposition \ref{prop:exact-lemma}.

%
%
\begin{thm}\label{thm:equivariant}
Let $\pconf_n(\CC)/\CC^\times$ be the quotient of pure configuration space by the free $\CC^\times$ action. The symmetric group 
$S_n$ acts on $\pconf_n(\CC)/\CC^\times$ by permuting coordinates. Let $A_n^\bullet$ be the sequence of $S_n$-modules constructed in 
Proposition \ref{prop:exact-lemma}. Then for each $k\geq 0$ we have  an isomorphism of $S_n$-modules
\[
	H^k(\pconf_n(\CC)/\CC^\times,\QQ)\cong A_n^k.
\]
\end{thm}

\begin{proof}
We regard $X_n := \pconf_n(\CC)$ as the total space of a $\CC^{\times}$-bundle over the base 
space $Y_n := \pconf_n(\CC)/\CC^\times$. 
As noted in Section \ref{subsec:twist} the cohomology of $X_n$ is that of the pure braid group, with its $S_n$-action.
Viewing $\CC^{\times}$ as $\RR^+ \times S^1$, 
 we see that $X_n$  
is  an $\RR^+$-bundle over the base space
$Z_n := \pconf_n(\CC)/\RR^+$, such that $Z_n$ is an $S^1$-bundle over $Y_n$.
The space $Z_n$ is a real-analytic manifold  which inherits the $S_n$-action.
For any $(z_1,z_2, \ldots, z_n) \in \pconf_n(\CC)$, let $[[z_1,z_2, \ldots, z_n]]$ denote its image in $Z_n$
Since $z_1 \neq z_2$, we may
rescale this vector by $c = \frac{1}{|z_1- z_2|} \in \CC^{\times}$ to get
$(\tilde{z}_1, \tilde{z}_2, \ldots, \tilde{z}_n) = \frac{1}{|z_1 - z_2|} (z_1, \ldots, z_n)$, which
comprise exactly the set of all $(\tilde{z}_1, \tilde{z}_2, \ldots, \tilde{z}_n) \in X_n$ 
satisfying the linear constraint  $\tilde{z}_1 - \tilde{z}_2 \in U(1) = \{ z \in \CC: \, |z|=1\}$.
We obtain a global section $Z_n \to X_n$ by mapping $[[z_1,z_2, \ldots, z_n]] \mapsto \frac{1}{|z_1 - z_2|}(z_1, \ldots, z_n)$,
so may regard $Z_n \subset X_n$, noting that it is invariant under the $S_n$-action.  Under this embedding we see that $Z_n$ is
a strong deformation retract of $X_n$, so has the same homotopy type as $X_n$.
The retraction map is: 
\[
	h_t(z_1, z_2,\ldots, z_n) := \big((1-t) |z_1 -z_2| + t\big) \frac{1}{|z_1-z_2|} (z_1, z_2, \ldots, z_n)  \quad  \mbox{for} \quad 0 \le t\le 1. 
\]
Consequently $H^k( X_n, \QQ) \cong H^k(Z_n, \QQ)$,
 for each $k \ge 0$ as $S_n$-modules.

For any $(z_1,z_2, \ldots, z_n) \in X_n$, let $[z_1,z_2, \ldots, z_n]$ denote its image in $Y_n$.
Since $z_1 \neq z_2$, we may
rescale this vector by $\frac{1}{z_1- z_2} \in \CC^{\times}$ to get
$(\tilde{z}_1, \tilde{z}_2, \ldots, \tilde{z}_n) = \frac{1}{z_1 - z_2}(z_1, \ldots, z_n)$, which
comprise exactly the set of all $(\tilde{z}_1, \tilde{z}_2, \ldots, \tilde{z}_n) \in X_n$ 
satisfying the linear constraint  $\tilde{z}_1 - \tilde{z}_2 = 1$.
These define a global coordinate system for $Y_n$, identifying it as an open complex manifold, 
and the  map $Y_n \to X_n$ sending $[z_1, z_2 \ldots, z_n] \mapsto (\tilde{z}_1, \tilde{z}_2, \ldots, \tilde{z}_n)$
is a nowhere vanishing global section of this bundle, so we may view  $Y_n \subset Z_n \subset X_n$. 
This map is a nowhere vanishing section of $Y_n$ inside the $S^1$-bundle $Z_n$ as well.

The Gysin long exact sequence for $Z_n$ as an $S^1$-bundle over $Y_n$ is 
\[
	\xrightarrow{e_\wedge} H^{k}(Y_n, \QQ) \rightarrow H^k(Z_n, \QQ) \rightarrow H^{k-1}(Y_n, \QQ) 
	\xrightarrow{e_\wedge} H^{k+1}(Y_n, \QQ) \rightarrow H^{k+1}(Z_n, \QQ) \rightarrow 
\]
The Euler class $e \in H^2(Y_n, \QQ)$ of this is zero since the bundle has a nowhere vanishing global section in $Z_n$. 
Thus $e_\wedge$ is the zero map and the  Gysin sequence 
splits into short exact sequences
\[
	0 \longrightarrow H^{k}(Y_n, \QQ) \longrightarrow H^k(Z_n, \QQ) \longrightarrow H^{k-1}(Y_n, \QQ) \longrightarrow  0.
\]
The maps are $S_n$-equivariant, since the Gysin sequence is
functorial.  
It follows from Maschke's theorem that
\begin{equation}
\label{sum decomp}
	H^k(X_n,\QQ) \cong H^k(Z_n, \QQ) \cong H^{k-1}(Y_n,\QQ) \oplus H^k(Y_n,\QQ)
\end{equation}
as $S_n$-modules.
Since $H^{-1}(Y_n,\QQ) = A_n^{-1} = 0$ by convention, we have $H^0(Y_n,\QQ) \cong A_n^0 \cong H^0(Z_n,\QQ) \cong H^0(X_n, \QQ)$.
It then follows inductively from \eqref{sum decomp} and
\[
	H^k(X_n,\QQ) \cong A_n^{k-1} \oplus A_n^k,
\]
that $H^k(Y_n,\QQ) \cong A_n^k$ as $S_n$-modules for all $k\geq 0$.
 \end{proof}

\begin{remark}\label{rmak}
The configuration space $\pconf(\CC)$ is  a hyperplane complement as treated in
the book of Orlik and Terao \cite{OrlikT:1992}. It equals
\[
	M(\sA_n): = \CC^n \smallsetminus \bigcup_{H_{i,j} \in \sA_n} H_{i,j},
\]
where 
$\sA_n :=\{ H_{i,j}: \, 1 \le i  < j \le n\}$ denotes  the {\em braid arrangement} of
hyperplanes 
$H_{i,j} : z_i = z_j$ for $1 \le i< j\le n$. 
\end{remark}

%
%
\section{Polynomial splitting measures and characters}\label{sec:splitting-characters}

We now express the splitting measure coefficients $\alpha_n^k(C_\lambda)$ in terms of the character values $\chi_n^k(\lambda)$ 
where $\chi_n^k$ is the character of the $S_n$-representation $A_n^k$ constructed in Proposition \ref{prop:exact-lemma}.
As a corollary we deduce that the rescaled $z$-splitting measures are characters when $z = -\frac{1}{m}$ and virtual characters when $z = \frac{1}{m}$, generalizing results from \cite{Lagarias:2016}.

%
%
\subsection{Expressing splitting measure coefficients by characters}\label{sec:51}

Recall,
\[
\nu_{n, z}^{\ast}(C_{\lambda}) =    \frac{N_\lambda(z)}{z^n - z^{n-1}}=    \sum_{k=0}^{n-1} \alpha_n^k(C_{\lambda}) \big(\tfrac{1}{z}\big)^k.
\]
We now express the splitting measure coefficient $\alpha_n^k(C_{\lambda})$ in terms of the character value $\chi_n^k(\lambda)$.

\begin{thm}\label{thm:splitting-coeffs}
Let $n\geq 2$ and $\lambda$ be a partition of $n$, then
\[
	\nu^*_{n,z}(C_\lambda) = \frac{1}{z_\lambda} \sum_{k=0}^{n-1}{(-1)^k \chi_n^k(\lambda) \big(\tfrac{1}{z}\big)^k},
\]
where $\chi_n^k$ is the character of the $S_n$-representation $A_n^k$ defined in Proposition \ref{prop:exact-lemma}. 
Thus,
\[
	\alpha_n^k(C_{\lambda}) = \frac{1}{z_\lambda}(-1)^k \chi_n^k(\lambda).
\]
\end{thm}

\begin{proof}
In Theorem \ref{cycle-coeffs} we showed
\[
	N_\lambda(z) = \frac{1}{z_\lambda}\sum_{k=0}^{n}{(-1)^k h_n^k(\lambda)z^{n-k}},
\]
where $h_n^k$ is the character of $H^k(P_n,\QQ)$. The $S_n$-representations $A_n^k$ were defined in 
Proposition \ref{prop:exact-lemma} where we showed that
\begin{equation}\label{dir sum}
	H^k(P_n,\QQ) \cong A_n^{k-1} \oplus A_n^k.
\end{equation}
Taking characters in \eqref{dir sum} gives
\[
	h_n^k = \chi_n^{k-1} + \chi_n^k.
\]
We compute
\begin{align*}
	\frac{N_\lambda(z)}{z^n} &= \frac{1}{z_\lambda}\sum_{k=0}^n{(-1)^kh_n^k(\lambda) \big(\tfrac{1}{z}\big)^k}\\
	&= \frac{1}{z_\lambda}\sum_{k=0}^n{(-1)^k\big(\chi_n^{k-1}(\lambda) + \chi_n^k(\lambda)\big)\big(\tfrac{1}{z}\big)^k}\\
	&= \big(1 - \tfrac{1}{z}\big)\frac{1}{z_\lambda}\sum_{k=0}^{n-1}{(-1)^k \chi_n^k(\lambda)\big(\tfrac{1}{z}\big)^k}.
\end{align*}
Dividing both sides by $\big(1 - \frac{1}{z}\big)$ yields
\[
	\nu_{n,z}^\ast(C_\lambda) = \frac{N_\lambda(z)}{\big(1 - \tfrac{1}{z}\big)z^n} = \frac{1}{z_\lambda}\sum_{k=0}^{n-1}{(-1)^k \chi_n^k(\lambda)\big(\tfrac{1}{z}\big)^k}.
\]
Comparing coefficients in the two expressions for $\nu_{n,z}^\ast(C_\lambda)$ we find
\[	
	\alpha_n^k(C_{\lambda}) = \frac{1}{z_\lambda}(-1)^k \chi_n^k(\lambda).
\]
\end{proof}
%
%
\subsection{Cycle polynomial and splitting measure as equivariant Poincar\'{e} polynomials} \label{sec:hyperplane2}

Given a complex manifold $X$, the \emph{Poincar\'{e} polynomial of $X$} is defined by
\[
	P(X, t) = \sum_{k\geq 0}{\dim H^k(X,\QQ) t^k}.
\]
If a finite group $G$ acts on $X$, then the cohomology $H^k(X,\QQ)$ is a $\QQ$-representation of $G$ with character $h_X^k$, and the \emph{equivariant Poincar\'{e} polynomial of $X$ at $g\in G$} is defined by
\[
	P_g(X,t) = \sum_{k\geq 0}\tr(g, H^k(X, \QQ) t^k = \sum_{k \ge 0} h_X^k(g)  t^k.
\]
Note that if $g = 1$ is the identity of $G$, then $h_X^k(1) = \dim H^k(X,\QQ)$ and $P_1(X,t) = P(X,t)$.

Under the change of variables $z = -\frac{1}{t}$,
the work  of Lehrer \cite[Theorem 5.5]{Lehrer:1987} identifies 
(rescaled)  cycle polynomials with  equivariant Poincar\'{e} polynomials of $\pconf_n(\CC)$, 
for $g \in S_n$, as 
\[
	\frac{1}{z^n}N_{[g]}(z) = \frac{|C_{\lambda}|}{n!}\sum_{k\geq 0}{h_n^k(g) t^k} = \frac{1}{z_\lambda}P_{g}(\pconf_n(\CC), t)
\]

Using the result of Section \ref{sec:43a} we obtain a similar interpretation of the splitting measure values. 

\begin{thm}\label{thm:splitting-coeffs2} 
Let $Y_n= \pconf_n(\CC)/\CC^{\times}$.
 Setting $\w = -\frac{1}{z}$,  for each $g \in S_n$ the $z$-splitting measure  is given
by the scaled equivariant Poincar\'{e} polynomial 
\[
	\nu^*_{n,z}(g) = \frac{1}{n!} \sum_{k=0}^{n-1}{ \tr(g: H^k( Y_n, \QQ)) \w^k},
\]
attached to  the complex manifold  $Y_n$, where $g$
acts as a permutation of the coordinate.
\end{thm}


\begin{proof} 
This formula follows  from Theorem \ref{thm:splitting-coeffs}, using also the identification of $A_n^k = H^k(Y_n, \QQ)$
as an $S_n$-module in Theorem \ref{thm:equivariant}.
Since we evaluate the character on a single element $g \in S_n$, the prefactor
becomes $\frac{1}{z_{\lambda} c_{\lambda}} = \frac{1}{n!}.$
\end{proof}

\begin{remark}
In the theory of hyperplane arrangements treated in \cite{OrlikT:1992} the  change of variable $z= -\frac{1}{t}$ appears as an involution
converting the Poincar\'{e} polynomial of a hyperplane complement (such as $\pconf_n(\CC)$) to another invariant,
the {\em characteristic polynomial} of an arrangement, given in  \cite[Defn. 2.52]{OrlikT:1992}).

\end{remark}
%
%

\subsection{Splitting measures for $z = \pm \frac{1}{m}$.}\label{sec:52}

Representation-theoretic interpretations of the rescaled $z$-splitting measures for $z = \pm 1$ were studied in \cite[Sec. 5]{Lagarias:2016}.
 Theorem \ref{thm:main-2a} below generalizes those results to give representation-theoretic interpretations for $z = \pm \frac{1}{m}$ when $m\geq 1$ is an integer.

\begin{thm}\label{thm:main-2a} 
Let $n\geq 2$ and $\lambda$ be a partition of $n$, then

\noindent(1) For $z = -\frac{1}{m}$ with $m\geq1$ an integer, we have
\[
	\nu_{n,-\frac{1}{m}}^\ast(C_\lambda) = \frac{1}{z_\lambda}\sum_{k=0}^{n-1}{\chi_n^k(\lambda)m^k}.
\] 
The function $z_\lambda \nu_{n, -\frac{1}{m}}^\ast(C_\lambda)$ is therefore the character of the $S_n$-representation
\[
	B_{n,m} = \bigoplus_{k=0}^{n-1}{\big(A_n^k\big)^{\oplus m^k}},
\]
with dimension
\[
	\dim B_{n,m} = \prod_{j=2}^{n-1}{(1 + jm)}.
\]
(2) For $z = \frac{1}{m}$ with $m\geq 1$ an integer, we have
\[
	\nu_{n, \frac{1}{m}}^\ast(C_\lambda) = \frac{1}{z_\lambda}\sum_{k=0}^{n-1}{(-1)^k\chi_n^k(\lambda) m^k}.
\]
The function $z_\lambda \nu_{n,\frac{1}{m}}^\ast(C_\lambda)$ is a virtual character, the difference of characters of representations $B_{n,m}^+$ and $B_{n,m}^-$,
\[
	B_{n,m}^+ \cong \bigoplus_{2j < n}{\big(A_n^{2j}\big)^{\oplus m^{2j}}} \hspace{2em} B_{n,m}^- \cong \bigoplus_{2j+1 < n}{\big(A_n^{2j+1}\big)^{\oplus m^{2j+1}}}.
\]
These representations have dimensions 
\[
	\dim B_{n,m}^{\pm} = \frac{1}{2} \Big( \prod_{j=2}^{n-1}(1+ jm)  \pm \prod_{j=2}^{n-1} (1-jm) \Big)
\]
respectively.
\end{thm}

\begin{proof}
(1) The formula for the $(-\frac{1}{m})$-splitting measure follows by substituting $z = -\frac{1}{m}$ in Theorem \ref{thm:splitting-coeffs}.
Arnol'd \cite[Cor. 2]{Arn:1969} shows the Poincar\'{e} polynomial $p(t)$ of the pure braid group $P_n$ has the product form
\[
	p(t) = \prod_{j=1}^{n-1}{(1 + jt)} = \sum_{k=0}^n{h_n^k\big((1^n)\big) t^k}.
\]
On the other hand, by Theorem \ref{cycle-coeffs} we have
\begin{equation}
\label{poincare}
	n! (-1)^n t^n N_{(1^n)}(-t^{-1}) = \sum_{k=0}^n{h_n^k\big((1^n)\big) t^k}.
\end{equation}
Dividing \eqref{poincare} by $1 + t$ we have
\begin{equation}
\label{poincare2}
	\prod_{j=2}^{n-1}{(1 + jt)} = n! (-1)^n t^n \frac{N_{(1^n)}(-t^{-1})}{1 + t} = \sum_{k=0}^{n-1}{\chi_n^k\big((1^n)\big) t^k}.
\end{equation}
Substituting $t = m$ gives the dimension formula.
 
(2) Substituting $z= \frac{1}{m}$ in Theorem \ref{thm:splitting-coeffs} gives the formula for the $(\frac{1}{m})$-splitting measure. Separating the even and odd parts we have
\[
	z_\lambda \nu_{n, \frac{1}{m}}^\ast(C_\lambda) = \sum_{2j < n}{\chi_n^{2j}(\lambda)m^{2j}} - \sum_{2j+1 < n}{\chi_n^{2j+1}(\lambda)m^{2j+1}}.
\]
Hence $z_\lambda \nu_{n,\frac{1}{m}}^\ast(C_\lambda) = \chi_{n,m}^+(\lambda) - \chi_{n,m}^-(\lambda)$, where $\chi_{n,m}^{\pm}$ are characters of $B_{n,m}^{\pm}$ respectively.
The dimension formulas follow from decomposing \eqref{poincare2} into even and odd parts.
\end{proof}

\begin{remark}
Other results in \cite[Theorems 3.2, 5.2 and 6.1]{Lagarias:2016} determine the values of the rescaled splitting measures for $z = \pm 1$,
showing they are supported on remarkably few conjugacy classes; for $z=1$ these were the Springer regular
elements of $S_n$. Theorem \ref{thm:main-2a} does not account for  the small support of the characters for $z =\pm1$.
The characters $h_n^k$ and $\chi_n^k$ have large support in general, hence cancellation must occur to explain the small support.
It would be interesting to account for this phenomenon.
\end{remark}

%
%

\subsection{Cohomology  of the pure braid group and the regular representation}\label{sec:53}

We use Theorem \ref{thm:splitting-coeffs} together with the splitting measure values at $z=-1$  computed in \cite{Lagarias:2016}
to determine a relation between the $S_n$-representation structure of the pure braid group cohomology
and  the regular representation of $S_n$. 
Let $A_n^k$ be the $S_n$-subrepresentation constructed in Proposition \ref{prop:exact-lemma}, and define the $S_n$-representation
\[
	A_n := \bigoplus_{k=0}^{n-1} A_n^k.
\]


\begin{thm}\label{thm:53}
Let  $\triv_n$, $\Sign_n$, and $\QQ [S_n]$ denote the trivial, sign, and regular representations of $S_n$  respectively. 
Then there are isomorphisms of $S_n$-representations,
\[
	\bigoplus_{k=0}^n{H^k(P_n,\QQ)\otimes \Sign_n^{\otimes k} }\cong\QQ[S_n].
\] 
and
\[
    A_n\otimes \big( \triv_n \oplus\Sign_n\big)\cong\QQ [S_n].
\]
\end{thm}


\begin{proof}
We showed in Proposition \ref{prop:exact-lemma} that $H^k(P_n,\QQ) \cong A_n^{k-1} \oplus A_n^k$, with
  $A_n^{-1} = A_n^n = 0$. Therefore, summing over  $0 \le k \le n$, 
\[
	A_n \cong \bigoplus_{k \text{ even}}H^k(P_n,\QQ) \cong \bigoplus_{k \text{ odd}}H^k(P_n,\QQ).
\]
Since $\Sign_n^{\otimes 2} \cong \triv_n$, we have
\begin{align*}
	\bigoplus_{k=0}^n{H^k(P_n,\QQ)\otimes \Sign_n^{\otimes k}}
	&\cong  \Big(\bigoplus_{k \text{ even}}H^k(P_n,\QQ)\otimes \triv_n \Big) \oplus \Big(\bigoplus_{k \text{ odd}}H^k(P_n,\QQ)\otimes \Sign_n \Big)\\
	&\cong (A_n \otimes \triv_n) \oplus (A_n \otimes \Sign_n)\\
	&\cong A_n \otimes (\triv_n \oplus \,\Sign_n).
\end{align*}
If $\chi_n$ is the character of $A_n$, then it follows from Theorem \ref{thm:main-2} that 
\[
	\chi_n(\lambda) = \sum_{k=0}^{n-1}\chi_n^k(\lambda) = z_\lambda \nu_{n,-1}^\ast(C_\lambda),
\]
so the values of $\chi_n$ are given by the rescaled $(-1)$-splitting measure. 

Theorem 6.1 of \cite{Lagarias:2016} shows
\[
	\nu_{n,-1}^\ast(C_\lambda) = \begin{cases} \frac{1}{2} & \lambda = (1^n) \text{ or } (1^{n-2}\, 2),\\ 0 & \text{otherwise.}\end{cases}
\]

Now let $\rho = \chi_n \cdot (1_n + \sign_n)$ be the character of $A_n \otimes (\triv_n \oplus\, \Sign_n)$. If $\lambda = (1^n)$, we compute
\[
	\rho(\lambda)= \chi_n(\lambda)\big(1 + \sign_n(\lambda)\big) = n! \nu_{n,-1}^\ast(C_\lambda)(2) = n!.
\]
If $\lambda = (1^{n-2}\, 2)$, then $\big(1 + \sign_n(\lambda)\big) = 0$, hence $\rho(\lambda) = 0$. If $\lambda$ is any other partition, then $\nu_{n,-1}^\ast(C_\lambda) = 0$, hence $\rho(\lambda) = 0$. Therefore $\rho$ agrees with the character of the regular representation, proving
\[
	\bigoplus_{k=0}^n{H^k(P_n,\QQ)\otimes \Sign_n^{\otimes k}} \cong A_n \otimes (\triv_n \oplus\, \Sign_n) \cong \QQ[S_n].
\]
\end{proof}

%
%
\section{Other Interpretations of $A_n^k$}\label{sec:stability}
Theorem \ref{thm:equivariant} interprets the $S_n$-representation $A_n^k$ geometrically as 
\[
	A_n^k \cong H^k(\pconf_n(\CC)/\CC^\times,\QQ).
\]

{ In this section we note two other interpretations of $A_n^k$, coming from combinatorial constructions previously
studied in the literature. These interpretations imply that the $A_n^k$ for fixed $k$ exhibit representation stability 
in the sense of Church, and Farb \cite{CF:2013} as $n \to \infty$.}\\

Proposition \ref{prop:exact-lemma} gave the following direct sum decomposition of the pure braid group cohomology,
\begin{equation}\label{eqn:A-iso}
	H^k(P_n, \QQ) \cong A_n^{k-1} \oplus A_n^k.
\end{equation}
The isomorphisms \eqref{eqn:A-iso} uniquely determine the $A_n^k$ as $S_n$-representations up to isomorphism.
Uniqueness holds since finite-dimensional representations are semisimple  by Maschke's theorem, using the general result that if
$0 = C^0, C^1, C^2,\ldots$ is any sequence of semisimple modules with submodules $B^k \subseteq C^k$, then isomorphisms
\[
	C^k \cong B^{k-1} \oplus B^k
\]
for each $k$ determine the $B^k$ up to isomorphism.

Let $\Pi_n$ denote the collection of partitions of a set with $n$ elements, partially ordered by refinement (see Stanley \cite[Example 3.10.4]{Stanley:1997}).

Hersh and Reiner \cite[Sec. 2]{HerRein:2015} describe two other sequences of $S_n$-representations giving direct sum decompositions
of $H^k(P_n, \QQ)$ coming from the Whitney and simplicial homology of the lattice $\Pi_n$. 
  %
 %

 \begin{prop}\label{prop:isomorphisms}
 
 (1) There is an isomorphism of $S_n$-representations
 \begin{equation}\label{whitney}
 H^k(P_n, \QQ) \cong WH_k(\Pi_n),
 \end{equation}
 where $WH_k(\Pi_n)$ is the $k$th Whitney homology of the lattice $\Pi_n$.
 
 (2) There is an isomorphism of $S_n$-representations 
\[
 WH_k(\Pi_n) \cong \beta_{ [k-1]} (\Pi_n) \oplus \beta_{ [k]} (\Pi_n) 
\]
where $\beta_{ [k]} (\Pi_n)$ is the $[k]= \{ 1, 2, ..., k\}$-rank selected homology of the lattice $\Pi_n$.

(3) There is an isomorphism of $S_n$-representations
\begin{equation*}
	\beta_{[k]}(\Pi_n) \cong \widetilde{H}_{k-1}\big(\Pi_n^k\big),
\end{equation*}
where $\Pi_n^k$ is the sub-poset of $\lambda \in \Pi_n$ with $|\lambda| - \ell(\lambda) \leq k$ and $\widetilde{H}_{k-1}\big(\Pi_n^k\big)$ 
denotes its reduced simplicial homology.
 \end{prop}
 
 \begin{proof}
 (1) This result is due to Sundaram and Welker \cite[Theorem 4.4 (iii)]{Sundaram-W:1997},
 cf.  \cite[Thm. 2.11, Sec. 2.3]{HerRein:2015}.
  (See \cite[Sec. 2.4]{HerRein:2015} for more on the Whitney homology of $\Pi_n$.) 
 
 (2)  Sundaram \cite[Prop. 1.9]{Sundaram:1994}
 decomposes $WH_k(\Pi_n)$ as
\begin{equation}\label{rank-selected}
	WH_k(\Pi_n) \cong \beta_{[k-1]}(\Pi_n) \oplus \beta_{[k]}(\Pi_n),
\end{equation}
where $[k] = \{1,2,\ldots,k\}$ and $\beta_{[k]}(\Pi_n)$ is the \emph{$[k]$-rank selected homology} of the lattice $\Pi_n$ \cite[Prop. 2.17]{HerRein:2015}. 

(3) Because the lattice $\Pi_n$ is \emph{Cohen-Macaulay}, Hersh and Reiner \cite[Sec. 2.5]{HerRein:2015} note the isomorphism
\begin{equation}\label{simplicial}
	\beta_{[k]}(\Pi_n) \cong \widetilde{H}_{k-1}\big(\Pi_n^k\big),
\end{equation}
where $\Pi_n^k$ is the sub-poset of $\lambda \in \Pi_n$ with $|\lambda| - \ell(\lambda) \leq k$ and $\widetilde{H}_{k-1}\big(\Pi_n^k\big)$ is its {\em reduced simplicial homology}.
 \end{proof} 

The following proposition relates $A_n^k$, $\beta_{[k]}(\Pi_n)$, and $\widetilde{H}_{k-1}\big(\Pi_n^k\big)$ using \eqref{eqn:A-iso}.

  %
 %
\begin{prop}\label{prop:connections}
 
Let $\Pi_n$ be the lattice of partitions of an $n$-element set, and $\Pi_n^k \subseteq \Pi_n$ the sub-poset comprised of $\lambda \in \Pi_n$ with $|\lambda| - \ell(\lambda) \leq k$. Then we have the following isomorphisms of $S_n$-representations
\[
	A_n^{k} \cong \beta_{[k]}(\Pi_n) \cong \widetilde{H}_{k-1}\big(\Pi_n^k\big).
\]
\end{prop}

\begin{proof}
The isomorphisms \eqref{whitney} and \eqref{rank-selected} in Proposition \ref{prop:isomorphisms}
 give the direct sum decompositions
\[ 
	H^k(P_n, \QQ) \cong \beta_{[k-1]}(\Pi_n) \oplus \beta_{[k]}(\Pi_n)
\]
for $ 0 \le k \le n$. By \eqref{eqn:A-iso} we have that  
\[
	H^k(P_n, \QQ) \cong A_n^{k-1} \oplus A_n^{k}.
\]
Since for $k=0$,
\[
	\beta_{[-1]}(\Pi_n) \cong A_n^{-1} = \{0\},
\]
we obtain by induction on $k \ge 1$ that
\[
	A_n^k \cong \beta_{[k]}(\Pi_n)
\]
Combining this isomorphism with \eqref{simplicial} finishes the proof.
\end{proof}

{ We deduce the representation stability  of the characters $\chi_n^k$ from known results.}

\begin{proof}[Proof of Theorem \ref{thm:main-3a}]
The $S_n$-representations of the rank-selected homology $\beta_{[k-1]}(\Pi_n)$ were shown by  Hersh and Reiner \cite[Corollary 5.4]{HerRein:2015}
to exhibit representation-stability for fixed $k$
and varying $n$ and to  stabilize sharply at $n=3k+1$ . This fact combined with
Proposition \ref{prop:connections} proves Theorem \ref{thm:main-3a}.
\end{proof}

{The following tables for $A_n^1$ and $A_n^2$ exhibit   representation stability and  the sharp stability phenomenon at  $n=3k+1$. We give
 irreducible decompositions, with multiplicities,  of $H^k(P_n, \QQ)$ and $A_n^1$  in Table \ref{table3a} and for $A_n^2$  in 

Table \ref{table3c}.
To read the  tables, for example, the entry $[4,1,1]$ denotes the isomorphism class of the irreducible representation of $S_6$ associated to the 
Specht module of the partition $[4, 1, 1]$ of $n=6$, in the notation of Sagan  \cite[Sec. 2.3]{sagan},
who gives  a construction of the Specht module representatives of the irreducible isomorphism classes.}

\begin{table}[h]
\renewcommand{\arraystretch}{1.2}
\begin{center}
	\begin{tabular}{| c | c |c | c|c|}
	\hline
	$n$  &  $\dim{H^1}$ & $H^1(P_n, \QQ)$  & $\dim A_n^1$ & $A_n^1$   \\ \hline
        $2$ & $1$ & $[2] $  & $0$ & $0$ \\
	$3$ & $3$ & $[3] \oplus [2,1]$  &$2$& $[2,1]$ \\
	$4$ & $6$ & $[4]\oplus [3,1] \oplus [2,2]$  &$5$& $[3,1] \oplus [2,2]$     \\
	$5$ & $10$  & $[5] \oplus [4,1] \oplus [3,2]$ & $9$& $[4, 1] \oplus [3, 2]$ \\
	\hline
         $n \ge 4$ & ${n \brack n-1}$  & $[n] \oplus [n-1,1] \oplus [n-2,2]$ & ${n \brack n-1} -1$& $[n-1,1] \oplus [n-2,2]$\\
	\hline
	\end{tabular}
	\medskip
	\caption{Irreducible $S_n$-module decompositions for $H^1(P_n, \QQ)$ and $A_n^1$.
	Here $\lambda$ abbreviates the irreducible representation $\sS^{\lambda}$.}
	\label{table3a}
\end{center}
\end{table}

\begin{table}[h]
\renewcommand{\arraystretch}{1.1}
\begin{center}
	\begin{tabular}{| c | c |c |}
	\hline
	$n$ &   $\dim A_n^2$&  $A_n^2$   \\ \hline
	$3$ & $0$& $0$ \\
	$4$ &  $6$ & $[3, 1] \oplus [2,1,1]$\\
	$5$ &  $26$ &  $[4,1] \oplus [3,2] \oplus 2[3,1,1] \oplus [2,2,1]$\\
	$6$ &  $71$ & $[5,1] \oplus [4,2] \oplus 2[4,1,1] \oplus [3,3] \oplus 2[3,2,1]$ \\
        $7$ &  $155$ & $[6,1] \oplus [5,2] \oplus 2[5,1,1] \oplus [4,3] \oplus 2[4,2,1] \oplus [3,3,1]$  \\
        $8$ & $295$ & $[7,1] \oplus [6,2] \oplus 2[6,1,1] \oplus [5,3] \oplus 2[5,2,1] \oplus [4,3,1]$ \\
        \hline
        $n \ge 7$ & ${n \brack n-2} - {n \brack n-1} +1$ & $[n-1,1] \oplus [n-2,2] \oplus 2[n-2,1,1] \oplus [n-3,3]$\\
        & &  $\oplus 2[n-3,2,1] \oplus [n-4,3,1]$\\
	\hline
	\end{tabular}
	\medskip
	\caption{Irreducible $S_n$-module decomposition for  $A_n^2$. }
	\label{table3c}
\end{center}
\end{table}

\paragraph{\bf Acknowledgments.} We  thank Richard Stanley for raising a question about the relation of the
braid group cohomology to the regular representation, answered by Theorem \ref{thm:main-2}.
We thank Weiyan Chen for pointing out to us that Theorem \ref{thm:main-0} is shown
in  Lehrer  \cite{Lehrer:1987} and for subsequently bringing the work of Gaiffi \cite{Gai:1996} to our attention.
We thank  Philip Tosteson and John Wiltshire-Gordon for helpful conversations. We thank the reviewers for
helpful comments. 


\end{document}